\documentclass[12pt%, draft
]{article}
\usepackage{amsmath,amsfonts,amssymb,amsthm,comment}

\newcommand{\1}{\mathbf {1}}

\DeclareMathOperator{\tr}{tr}
\DeclareMathOperator{\Conj}{Conj}

\newtheorem{theorem}{Theorem}[section]
\newtheorem{proposition}[theorem]{Proposition}
\newtheorem{lemma}[theorem]{Lemma}
\newtheorem{corollary}[theorem]{Corollary}
\theoremstyle{definition}

\newtheorem{conjecture}[theorem]{Conjecture}

\theoremstyle{remark}
\newtheorem{remark}[theorem]{Remark}
\newtheorem{example}[theorem]{Example}

\newcommand{\w}{\omega}
\DeclareMathOperator{\Ker}{Ker\,}
\DeclareMathOperator{\Irr}{Irr\,}

\newcommand\Z{\mathbb{Z}}
\newcommand\Zpos{\Z_{\geq0}}
\newcommand\Zplus{\Z_{>0}}

\newcommand\C{\mathbb{C}}
\newcommand\N{\mathbb{N}}

\newcommand{\CB}{{\mathcal B}}

\newcommand{\FX}{{ X}}
\newcommand{\FK }{{ K}}
\newcommand{\FD}{{ D}}
\newcommand{\FW}{{ W}}

\newcommand{\Dih}{{\rm D}}
\newcommand{\Qu}{{\rm Q}}
\newcommand{\SD}{{\rm SD}}
\newcommand{\NO}{\,{\raise0.25em\hbox{$\mathop{\hphantom {\cdot}}\limits^{_{\circ}}_{^{\circ}}$}}\,}

\newcommand\ol{\overline}

\numberwithin{equation}{section}

\begin{document}
\title{Harada\rq{}s conjecture II and Gramian determinants}

\author%[T. Abe]
{Toshiyuki Abe\thanks{abe.toshiyuki.mz@ehime-u.ac.jp, Faculty of Education, Ehime University,
Matsuyama, Ehime 790-8577, Japan}
}
\date{ }
%\address{Faculty of Education, Ehime University, Matsuyama, Ehime 790-8577, Japan}
%\email{abe.toshiyuki.mz@ehime-u.ac.jp}
%
%\subjclass[2010]{17B69, 17B65}

%\keywords{Group, Harada's conjecture, $p$-groups}

\maketitle

%\thanks{T.A. is partially supported by JSPS fellow ****. }

\abstract{
Let $G$ be a finite group. 
Harada's conjecture II states that the ratio of the product of all the number of elements in conjugacy classes over that of all degrees of irreducible complex characters of $G$ is an integer.
The ratio is called Harada's number. 
In this article, we discuss the Harada\rq{}s number from the view point of Gramian determinants for suitable inner spaces and introduce an invariants generalizing the square of Harada's number associated to central characters of $G$. 
We also give a necessary and sufficient condition so that  Harada's conjecture II holds.
We calculate explicitly Harada\rq{}s numbers in some examples by using the method given in this article.     
}

%Intro
\section{Introduction}
This is an article for the aim of the affirmative solution to Harada's conjecture II.  
The conjecture was presented in \cite{Harada} by Harada (It was also mentioned to Harada\rq{}s conjecture I in \cite{Harada}).
The conjecture relates between the sizes of conjugacy classes and the degrees of complex irreducible characters, and states that Harada\rq{}s number of a finite group is always an integer. 
This conjecture has been studied and verified for some concrete models or series of groups (see \cite{Harada}, \cite{Hida}, \cite{Kiyota}, \cite{Sugimoto}  and \cite{Miyamoto}). 
In this article, we consider Gramian determinants on a subspace of the center of the group algebras with suitable basis and give a necessary and sufficient condition to Harada\rq{}s conjecture II. 
By using the technique, we give alternative calculations of Harada\rq{}s numbers explicitly for some known example.    

We shall give a statement of Harada\rq{}s conjecture II. 
Let $G$ be a finite group, and consider the set $\{K_1,\cdots, K_s\}$ of all conjugacy classes and the set $\{\chi_1,\cdots, \chi_s\}$ of all complex irreducible characters.
{\em Harada\rq{}s number} is the ratio 
\[
h(G):=\frac{|K_1|\cdots |K_s|}{\chi_1(1)\cdots \chi_s(1)},
\]
where $1$ denotes the unit of a group. 
\begin{conjecture}(Harada\rq{}s conjecture II)
For any finite group $G$, $h(G)\in \Z$.
\end{conjecture}
In the case $G$ is abelian, $h(G)=1$ since $|K_i|=\chi_i(1)=1$ for all $i$. 
In the non abelian case, calculations of the numbers elements in conjugacy classes and degrees of irreducible characters are sufficient and hence the character table 
\[
X=(\chi_i(g_j))_{i,j=1,\cdots, s},
\]
of $G$ gives enough information to determine $h(G)$, where $g_i$ is a representative of $K_i$ for each $i$. 
In particular, it is verified that every finite simple group $G$ in \cite{ATLAS} satisfies $h(G)\in\Z$ by Chigira (as announced in \cite{Harada}). 
It is also proved that $h(G)$ is integer if $G$ is the symmetric groups and the alternating groups by Hida (\cite{Hida}) by using the well known correspondence of irreducible modules over $\C$ and conjugacy classes via Young diagrams. 
In another direction, Kiyota proved that $h(G)\in \Z$ for finite groups $G$ whose every Sylow subgroup is abelian (see \cite{Kiyota}) in help of  Brauer\rq{}s height $0$ conjecture. 
Recently, it has been proved for a some series of finite groups of Lie type (\cite{Sugimoto}) and some groups of nilpotent class $2$ (\cite{Miyamoto}).   

%In the representation theory of finite groups, the number $\omega_{\chi_i}(g_j)=\chi_i(g_j)|K_j|/\chi_i(1)$ appears as the eigenvalue of the action of $g_i$ on the irreducible $\C G$-module affording the character $\chi_i$. 
%Harada's number $h(G)$ appears in the determinant of the matrix$(\omega_{\chi_i}(g_j))_{i,j=1,\cdots, s}$ which is $h(G)\sqrt{|\det(X)|}$ as pointed out by Harada in \cite{Harada}. 
%
%We treat this observation in the context of Hermitian form on the center of the group algebra $\C G$ as . 
%

One of the main result of this article is that Harada\rq{}s number of $G$ can be expressed by means of the structure of the center $Z(\C G)$ of the the usual group algebra $\C G$. 
We write $[g]$ for the vector in $\C G$ corresponding to $g\in G$, and for any subset $U\subset G$, write $[U]=\sum_{x\in U}[x]$. 
Then the ordered set of the class sums $[K_1],\cdots, [K_s]$ forms a basis of $Z(\C G)$. 
It is also well known that 
\[
e_{\chi_i}=\frac{\chi_i(1)}{|G|}\sum_{j=1}^{s}\chi_i(g_j^{-1})[K_j]
\] 
for $i=1,\cdots, s$ are primitive idempotents and that they also forms a basis of $Z(\C G)$. 
For any $i,j$, $[K_j]e_{\chi_i}=\omega_{ij} e_{\chi_i}$ with constant   
\[
\omega_{ij}=\frac{\chi_i(g_j)|K_j|}{\chi_i(1)}.
\]
The determinant of the matrix $W=(\omega_{ij})_{i,j=1,\cdots,s}$ is $\det W=h(G) \det X$ as pointed out by Harada in \cite{Harada}. 

We note that $|\det W|^2=\det({}^t \ol{W} W)=h(G)^2 c(G)$, where 
\[
c(G)=\det({}^t\ol{X} X)=|C_G(g_1)|\cdots|C_G(g_s)|.
\]
Hence $h(G)$ appears in the absolute value of a Hermitian matrix ${}^t \ol{W} W$, that is  Gramian matrix of a certain Hermitian space over $\C$. 
As a desired Hermitian space, we can choose the algebra $Z(\C G)$. 
We see that $Z(\C G)$ becomes a Hermitian space equipped with a Hermitian form defined by  
\[
\langle u|v\rangle :=\tr_{Z(\C G)} \overline{u}v 
\]
for $u, v\in Z( \C G)$, where $\tr_{Z(\C G)} u$ is the trace of the linear map given by  $w\mapsto u w$ for $w\in Z(\C G)$ and the map $\overline{\cdot}$ is a skew-linear map which $[g]$ sends to $[g^{-1}]$ for $g\in G$ 
For any subset $\CB$ of $Z(\C G)$, we consider the absolute value $\gamma(\CB)$ of the determinant of the Gramian matrix $(\langle u|v\rangle)_{u,v\in \CB}$;  
\[
\gamma(\CB)=|\det((\langle u|v\rangle)_{u,v\in \CB})|.   
\]
%Then $\gamma(\CB)$ does not depend on the order of vectors in $\CB$.   
Then we can show that   
\[
\gamma(G):=\gamma(\{[K_1],\cdots, [K_s]\})\in \Z
\]  
is nonzero and admits with $\det({}^t\ol{W} W)$. 
Since $\gamma(G)$ and $c(G)$ are integers, we have the following criterion.
\begin{theorem} {\rm (Theorem \ref{thm3-1})}
Let $G$ be a finite group.
Then $\gamma(G)=h(G)^2c(G)$. 
Therefore, $h(G)\in \Z$ if and only if $c(G)$ divides $\gamma(G)$.  
\end{theorem}

This theorem implies that Harada\rq{}s conjecture II can be verified by using only the group structure of $G$, not using the classification of irreducible characters. 
We consider a number $\mu(G):=h(G)^2=\gamma(G)/c(G)$ for a finite group $G$.
We generalize the number $\mu(G)$ to $\mu_{\phi}(G)$ for any irreducible character $\phi\in \Irr(Z)$ of a central subgroup $Z$ of $G$. 
For any $\phi\in \Irr(Z)$, we consider an idempotent $e_\phi=\frac{1}{|Z|}\sum_{z\in Z}\phi(z^{-1})[z]\in Z(\C G)$.
Then we take a suitable basis $\CB(G)_\phi\subset \{e_\phi[K_1],\cdots, e_\phi[K_s]\}$ of $e_\phi Z(\C G)$ and define $\gamma_\phi(G): =\gamma(\CB(G)_\phi)$.
We also define $c_\phi(G)$ as a generalization of $c(G)$. 
Then we define $\mu_\phi(G)=\gamma_\phi(G)/c_\phi(G)$. 
The number $\mu_\phi(G)$ admits $\mu(G)$ in the case $\phi=\1$ is the trivial character of $Z=\{1\}$. 
We also prove that 
\[
\mu(G)=\prod_{\phi\in \Irr(Z)}\mu_\phi(G)
\] 
for any central subgroup $Z$ of $G$ as in Theorem \ref{thm3-7}.  
Unfortunately, we have examples where $\mu_\phi(G)$ is not an integer for some faithful $\phi\in \Irr(Z)$.
We also consider the case $\Ker\phi=N$ is not $\{1\}$.
In this case, $\phi$ induces a faithful character $\ol{\phi}$ of $Z/N$ and we show that $\mu_\phi(G)=\kappa_N(G)^2\mu_{\ol{\phi}}(G/N)$ for some integer $\kappa_N(G)$ defined for the subgroup $N$ of $Z$. 

One of the aim of this article is to demonstrate calculations of $\gamma_\phi(G)$ and $\mu_\phi(G)$ for some finite groups $G$. 
In fact, we do in the case $G$ is an abelian group, a dihedral group, a semi-dihedral group, a quaternion group, a $p$-group which contains a maximal cyclic subgroup for $p>2$, and central products of finite groups.     
%For some of examples we give here, we can calculate $\mu(G)$ or $\mu_\phi(G)$ in more easier way. 
%We will discuss the method in a forthcoming paper.    

This article is organized as follows. 
In Section 2.1, we give a preliminary of Hermitian spaces and Gramian determinants. 
In Section 2.2, we state Harada's conjecture II. 
In Section 2.3, we give a criterion to Harada's conjecture.  
Section 3 is devoted to a generalization of the square of Harada's number for central characters.  
In Section 3.1, we generalize $\mu(G)$ to $\mu_\phi(G)$ for a characters of a central subgroup. 
In Section 3.2, we show the factorization property of $\mu(G)$. 
In Section 3.3, we consider the relation of $\mu_\phi(G)$ and $\mu_{\ol{\phi}})(G/\ker\phi)$ where $\ol(\phi)$ is the faithful character induced from $\phi$. 
We give some examples in Section 4.   
%\medskip

\paragraph{Acknowledgment:} 
The author thanks Naoki Chigira for stimulating discussions and comments. 
He also express a gratitude to Masao Kiyota and Akihiro Hida for the discussions. 
This work is partially supported by JSPS fellow 19K03403.  

%\tableofcontents

\section{Preliminaries and Harada\rq{}s conjecture II}
In this paper, we only consider finite groups. 
The units of groups are denoted by $1$. 
In a group $G$, we  use the notations 
\[
{}^xy:=xyx^{-1}, \quad y^{x}=x^{-1}yx
\] 
for elements $x,y\in G$. 
For simplicity, when $H$ is a subgroup of a group $G$, we write $H\leq G$.   
We denote by $G\rq{}$ and $Z(G)$ the commutator subgroup and the center of $G$, respectively. 
We write for $\C G$ the group algebra of $G$ over a complex number field $\C$. 
We write $|U|$ for the number of elements in a finite set $U$. 
\subsection{Some notations with respect to a Hermitian space}
%Let $F$ be a subfield of $\C$ closed by the complex conjugation.  
Let $A$ be a finite dimensional semi-simple associative algebra over $\C$ equipped with a skew-linear algebra homomorphism $\overline{\cdot}:A\rightarrow A$, that is, a ring endomorphism satisfying $\overline{\lambda a}=\overline{\lambda}\overline{a}$ for $a\in A$ and $\lambda\in \C$.
Let $Z(A)$ be the center of $A$, that is, 
\[
Z(A)=\{u\in A|a u=u a\text{ for }a\in A\}. 
\]
Then we have a Hermitian form on $Z(A)$ defined by 
\[
\langle u|v\rangle:=\tr_{Z(A)} \overline{u} v,\quad u,v\in Z(A), 
\]
where $\tr_{Z(A)} u$ is the trace of the linear map given as the multiplication of $u$ on $Z(A)$. 
We note that $\langle\cdot |\cdot\rangle$ is non-degenerate. %Hermitian form on $ Z(A)$ which is skew-linear for the first entry. 

For an ordered set $\CB=\{v_1,\cdots, v_d\}$ of vectors in $Z(A)$, we consider Gramian matrix  
\[
\Gamma(\CB)=(\langle v_i|v_j\rangle)_{i,j=1,\cdots, d}
\]   
and denote by $\gamma(\CB)$ the absolute value of the determinant of Gramian matrix $\Gamma(\CB)$: 
\[
\gamma(\CB)=|\det(\Gamma(\CB))|. 
\] 
%We call the value the {\em discriminant} of $\CB$.  
It is clear that $\gamma(\CB)$ does not depend on the order of vectors in $\CB$ and that $\CB$ is linearly independent if and only if $\gamma(\CB)$ is nonzero.

For a linearly independent ordered subset $\CB_1=\{v_i|i=1,\cdots, d\}$ of $Z(A)$ and an ordered subset $\CB_2=\{u_i|i=1,\cdots, d\}$ of the subspace $\langle \CB_1\rangle_{\C}$ spanned by $\CB_1$, we have a square matrix $T=T_{\CB_1,\CB_2}=(t_{ij})$ subject to 
\[
u_i=\sum_{j=1}^{d}t_{ji}v_j,\quad i=1,\cdots, d.   
\] 
We call the absolute value of the determinant of $T$ the {\em index} of   $\CB_2$ for $\CB_1$ and denote by $[\CB_2:\CB_1]$:
\[
[\CB_2:\CB_1]=|\det(T)|. 
\] 
The index does not depend on the order of vectors in $\CB_1$ and $\CB_2$. 
In the case $\CB_2$ is linearly independent, $\CB_2$ is a basis of $\langle \CB_1\rangle_\C$ and $[\CB_2:\CB_1]\neq 0$. 
Thus $\CB_1$ is a basis of $\langle \CB_2\rangle_{\C}$ and $[\CB_1:\CB_2]=[\CB_2:\CB_1]^{-1}$.  
Since $\Gamma(\CB_2)={}^t\overline{T}\,\Gamma(\CB_1)\,T$, we have 
\begin{align}\label{eqn3-2}
\gamma(\CB_2)=[\CB_2:\CB_1]^2\gamma(\CB_1). 
\end{align}
%We shall apply the procedure of discriminants to the center of the group algebra of a finite group in the next section. 

%%
\subsection{Harada\rq{}s conjecture II and a criterion for the conjecture}
Let $G$ be a finite group, $\Irr(G)=\{\chi_1,\cdots, \chi_s\}$ the set of all irreducible characters of $G$ valued in $\C$ and $\Conj(G)=\{K_1,\cdots, K_s\}$ that of all conjugacy classes of $G$. %with $s:=|\Conj(G)|=|\Irr(G)|$. 
The rational number 
\[
h(G):=\frac{\prod_{i=1}^{s}|K_i|}{\prod_{i=1}^{s}\chi_i(1)}
\] 
is called {\em Harada\rq{}s number} of $G$. 
\begin{conjecture}{\rm (Harada's conjecture II)} {\rm (\cite{Harada})}
Let $G$ be a finite group. 
Then $h(G)\in \Z$. 
\end{conjecture}  
It is also known a stronger conjecture. 
\begin{conjecture}{\rm (Harada-Chigira's conjecture)} 
Let $G$ be a finite group and $G\rq{}$ its commutator subgroup.  
Then $%h\rq{}(G):=
h(G)/|G\rq{}|\in \Z$. 
\end{conjecture}   
%We refer Harada\rq{}s conjecture II and Harada-Chigira\rq{}s conjecture as (H) and (HC), respectively.   
%It is clear that if Harada-Chigira\rq{}s conjecture is valid then so is Harada\rq{}s conjecture II. 

We fix a representative $g_i\in K_i$ for each $i$.
Then the matrix
\[
{\FX}=(\chi_i(g_j))_{i,j=1,\cdots, s} 
\]
is called the {\em character table} of $G$. 
The orthogonality of irreducible characters proves    
\[
{}^t\overline{\FX} \FX=(\delta_{i,j}|C_G(g_i)|)_{i,j=1,\cdots, s},  
\]
where $C_G(g_i)$ is the centralizer of $g_i$ in $G$.
Hence we have  
\begin{align}
c(G):=\prod_{i=1}^{s}|C_G(g_i)|=\det({}^t\overline{\FX} \FX). 
\end{align}
%In particular, $\CB_{\rm conj}:=\{[K_1],\cdots, [K_s]\}$ forms a basis of $Z(\C G)$. 
%\subsection{Discriminants and Harada\rq{}s conjecture II}

We consider diagonal matrices  
\[
\FK=(\delta_{i,j}|K_i|)_{i,j=1,\cdots, s},\quad \FD=(\delta_{i,j}\chi_{i}(1))_{i,j=1,\cdots, s}. 
\]
Then $h(G)=\det(\FK \FD^{-1})$. 
For $i,j=1,\cdots, s$, we define 
\[
\omega_{\chi_i}(g_j):=\frac{\chi(g_j)|K_j|}{\chi_i(1)}. 
\]
and consider a matrix 
\begin{align*}
\FW:=(\omega_{\chi_i}(g_{j}))_{i,j=1,\cdots, s}=\FD^{-1} \FX \FK. 
\end{align*}
Then  
\begin{align*}
{}^t\overline{\FW} \FW=\FK{}^t\overline{\FX} \FD^{-2} \FX \FK. 
\end{align*}
In particular, we have 
\begin{align*}
\det({}^t\overline{\FW} \FW)=h(G)^2c(G). 
\end{align*}
%We shall show that $\Gamma(\CB_{\rm conj})={}^t\overline{\Omega}\Omega$ is obtained as a Gramian matrix of some lattice in $Z(\C G)$. 
%Let consider the group algebra $A=\C G$ and its center $Z(\C G)$. 

To express $\det({}^t\overline{\FW} \FW)$ as Gramian determinant, we consider the center of the group algebra $\C G$. 
The algebra $\C G$ of $G$ is a vector space over $\C$ with basis $\{[g]|g\in G\}$ whose multiplication is given by $[g][h]=[gh]$ for $g,h\in G$. 
For any subset $S\subset G$, we write $[S]=\sum_{g\in S}[g]$. 
Then the set 
\[
\CB_{\rm conj}(G)=\{[K]|K\in \Conj(G)\}
\]
forms a basis of $Z(\C G)$.
On the other hand, for any $\chi\in \Irr(G)$, 
\begin{align}\label{eqn2-2-2}
e_{\chi}=\frac{\chi(1)}{|G|}\sum_{g\in G}\overline{\chi(g)}[g]=\frac{\chi(1)}{|G|}\sum_{i=1}^{s}\overline{\chi(g_i)}[K_i]
\end{align}
becomes a primitive idempotent of $\C G$.
Furthermore, $e_{\chi}e_{\mu}=\delta_{\chi,\mu}e_{\chi}$ for any $\chi,\mu\in \Irr(G)$ and the set $\{e_{\chi}|\chi\in\Irr(G)\}$ also forms a basis of $Z(\C G)$.   
It is known that 
\begin{align}\label{eqn2-45}
[K_j]e_{\chi_i}=\omega_{\chi_i}(g_j) e_{\chi_i}
\end{align}
for any $i,j=1,\cdots,s$.  

Now we define $\overline{\cdot}:\C G\rightarrow \C G$ by 
\[
\overline{\sum_{g\in G}c_g [g]}=\sum_{g\in G}\overline{c_g} [g^{-1}], \quad c_g\in \C. 
\]
It is easy to see that $\overline{\cdot}$ is a semi-linear anti-automorphism and that  
\[
\overline{e_{\chi_i}}=e_{\chi_i},\quad \overline{[K_i]}=[K_{i^*}]
\]
for $i=1,\cdots, s$, where we write $K_{i^*}=\{g^{-1}|g\in K_i\}$. 
For the  basis $\CB_{\rm conj}(G)$ of $Z(\C G)$, we define    
\[
%\Gamma(G):=\Gamma(\CB_{\rm conj}), \quad 
\gamma(G):=\gamma(\CB_{\rm conj}(G)). 
\]
As one of the main results of this article, we give a criterion for Harada\rq{}s conjecture II.  
\begin{theorem}\label{thm3-1}
Let $G$ be a finite group.
Then $\gamma(G)\in \Z$ and $\gamma(G)=h(G)^2c(G)$. 
Therefore, $h(G)\in \Z$ if and only if $c(G)$ divides $\gamma(G)$.  
\end{theorem}
\begin{proof}
By \eqref{eqn2-45},  
\[
\overline{[K_i]}{[K_j]}e_{\chi_k}=\w_{\chi_k}(g_{i}^{-1})\w_{\chi_k}(g_{j})e_{\chi_k}=\overline{\w_{\chi_k}(g_{i})}\w_{\chi_k}(g_{j})e_{\chi_k}
\] 
for $i,j,k=1,\cdots, s$. 
Thus we have 
\[
\langle [K_i]|[K_{j}]\rangle=\sum_{k=1}^{s} \overline{\w_{\chi_k}(g_i)}\w_{\chi_k}(g_j). 
\] 
This is the $(i,j)$-entry of the matrix ${}^t\overline{\FW} {\FW}$.
Hence $\Gamma(\CB_{\rm conj}(G))={}^t\overline{\FW} \FW$ and 
\[
\gamma(G)=|\det\Gamma(\CB_{\rm conj}(G))|=|{}^t \overline{\FW}\FW|=h(G)^2c(G). 
\]

We consider the structure constants $\alpha_{ij}^k\in \Zpos$ of the algebra $Z(\C G)$ with respect to the basis $\CB_{\rm conj}(G)$ subject to 
\[
[K_i][K_j]=\sum_{k=1}^{s}\alpha_{ij}^k[K_k]
\] for $i,j,k=1,\cdots, s$. 
Then 
\[
\tr_{Z(\C G)}([K_k])=\sum_{l=1}^{s}\alpha_{kl}^{l}
\]
and 
\[
\langle[K_i]|[K_{j}]\rangle =\sum_{k=1}^{s}\sum_{l=1}^{s}\alpha_{i^*j}^{k}\alpha_{kl}^{l}\in \Zpos
\] 
for $i,j, k=1,\cdots, s$. 
Hence, $\Gamma(\CB_{\rm conj}(G))$ is an integer matrix and $\gamma(G)$ is also an integer. 
\end{proof}

We define 
\[
\mu(G):=\frac{\gamma(G)}{c(G)}
\]
for a finite group $G$. 
Then it follows from Theorem \ref{thm3-1} that $\mu(G)=h(G)^2$ and $h(G)\in\Z$ if and only if $\mu(G)\in \Z$.  

We can show the following proposition by suing the definition directly, 
\begin{proposition}\label{prop3-3}
Let $G,H$ be finite groups.
Then 
\[
\mu(G\times H)=\mu(G)^{|\Conj(H)|}\mu(H)^{|\Conj(G)|}.
\]  
\end{proposition}
\begin{proof}
Put $s=|\Conj(G)|$, $t=|\Conj(H)|$ and set ${\rm Conj}(G)=\{K_1,\cdots, K_s\}$ and ${\rm Conj}(H)=\{L_1,\cdots, L_t\}$.
Then 
\[
 {\rm Conj}(G\times H)=\{K_i\times L_j|i=1,\cdots, s, j=1,\cdots, t\}
\] 
and $ \CB_{\rm conj}(G)=\{[K_i\times L_j]|i=1,\cdots, s, j=1,\cdots, t\}$ 
is a basis of $Z(\C (G\times H))$. 
%For simplicity, set $M_{i,j}=K_i\times L_j$ for $i,j$, and define structure constants $\alpha_{ij}^k$ and $\beta_{ij}^k$ by $[K_i][K_j]=\sum_{k}\alpha_{ij}^k[K_k]$, $[L_i][L_j]=\sum_{k}\beta_{ij}^k[L_k]$.
%Then $[M_{i,j}][M_{m,n}]=\sum_{k,l}\alpha_{im}^{k}\beta_{jn}^{l}[M_{k,l}]$. 
%Therefore, 
By direct calculations, 
\begin{align*}
\langle [K_{i}\times L_{j}]|[K_{m}\times L_{n}]\rangle 
%&=\sum_{k_1,l_1,k_2,l_2}(\alpha_{i^*m}^{k_1}\beta_{j^*n}^{k_2})(\alpha_{k_1l_1}^{l_1}\beta_{k_2l_2}^{l_2})\\
%&=\sum_{k_1,l_1}\alpha_{i^*m}^{k_1}\alpha_{k_1l_1}^{l_1}\sum_{k_2,l_2}\beta_{j^*n}^{k_2}\beta_{k_2l_2}^{l_2}\\
%&
=\langle [K_i]|[K_m]\rangle\langle [L_j]|[L_n]\rangle. 
\end{align*}
Hence 
\[
\Gamma(\CB_{\rm conj}(G\times H))=\Gamma(\CB_{\rm conj}(G))\otimes\Gamma(\CB_{\rm conj}(H)), 
\]
where $A\otimes B$ is the tensor product of the matrices $A$ and $B$. 
In particular,  one sees that $\gamma(G\times H)=\gamma(G)^t\gamma(H)^s$. 
Take representatives $g_i\in K_i$ and $h_j\in L_j$ for each $i$ and $j$. 
Then 
\[
C_{G\times H}((g_i,h_j))=C_G(g_i)\times C_H(h_j).
\] 
Hence 
\[
c(G\times H)=\prod_{i,j}(|C_G(g_i)||C_H(h_j)|)=c(G)^t c(H)^s
\] 
and we have $\mu(G\times H)=\mu(G)^t \mu(H)^s$. 
\end{proof}

\section{A generalization of Harada\rq{}s number} 
\subsection{A number $\mu_{\phi}(G)$  associated to characters $\phi$ of central subgroups of $G$}
Let $G $ be a finite group and $Z$ a subgroup of the center $Z(G)$. 
In this section, we shall define a number $\gamma_\phi(G)$ and an integer $c_\phi(G)$ for a character $\phi\in \Irr(Z)$.
As their ratio, we shall define $\mu_\phi(G)=\gamma_\phi(G)/c_\phi(G)$. 

Let $\phi\in \Irr(Z)$. 
Since $Z\leq Z(G)$, 
\begin{align}\label{char2-1}
e_{\phi}:=\frac{1}{|Z|}
\sum_{z\in Z}\overline{\phi(z)}[z] 
\end{align}
is an idempotent in $Z(\C G)$.
Since $\sum_{\phi\in \Irr(Z)}e_{\phi}=[1]$,   
\[
\C G=\bigoplus_{\phi\in \Irr(Z)} e_\phi \C G
\] 
is a direct sum of semi-simple ideals of $\C G$.
We note that $\{e_{\phi}|\phi\in \Irr(Z)\}$ is linearly independent and have 
\begin{align}\label{index2-1}
[\{e_{\phi}|\phi\in \Irr(Z)\}:\{[z]|z\in Z\}]^2=|Z|^{-2|Z|}|\det(\overline{\phi(z)})_{\phi\in \Irr(Z), z\in Z}|^2=|Z|^{-|Z|},
\end{align}
where the matrix $(\overline{\phi(z)})_{\phi\in \Irr(Z), z\in Z}$ is the character table of $Z$.
We see that 
\[
|\det(\overline{\phi(z)})_{\phi\in \Irr(Z), z\in Z}|^2=c(Z)=|Z|^{|Z|}.
\]

We regard $\Conj(G)$ as a $Z$-set via the action $(z, K)\mapsto zK$ for $K\in \Conj(G)$ and $z\in Z$.
The $Z$-orbits in $\Conj(G)$ are parameterized by the quotient $G/Z$ as follows. 
Let 
\[
\pi_Z:G\rightarrow G/Z,\quad g\mapsto gZ
\]
be the canonical projection and set $\Conj(G/Z)=\{\overline{K}_{i}|i\in I_{G/Z}\}$, where $I_{G/Z}$ is an index set of $\Conj(G/Z)$. 
%Here we set $I_{G/Z}=\{1,\cdots,|\Conj(G/Z)|\}$. 
Then take an element $g_i\in \pi_Z^{-1}(\overline{K}_i)$ for $i\in I_{G/Z}$ and set $K_i=g_i^G$. 
Then $\{K_i|i\in I_{G/Z}\}$ is a complete set of  representatives of $Z$-orbits in $\Conj(G)$.  
In particular, we have 
\begin{align}\label{conj2-1-1}
\Conj(G)=\{zK_j|z\in Z, j\in I_{G/Z}\}. 
\end{align}
%We note that  
%\begin{align}\label{index2-2}
%|I(G/Z)|=s(G/Z). 
%\end{align}
For any $K\in \Conj(G)$ and $N\leq Z$, set 
\begin{align*}
{\rm Ann}_{Z}(K):&=\{z\in Z|zK=K\},\\
 \Conj(G:Z)_N:&=\{K\in \Conj(G)|{\rm Ann}_Z(K)=N\}.  
\end{align*}
The set $\Conj(G:Z)_N$ is a $Z$-subset of $\Conj(G)$. 
Let  
\[
I(G:Z)_N=\{j\in I_{G/Z}|{\rm Ann}_Z(K_j)=N\}. 
\]
%to be an index set of a complete set of representatives in the $Z$-subset $\Conj(G:Z)_N$ of $\Conj(G)$ for any subgroup $N\leq Z$.
%We also denote 
%\[
%s(G:Z)_{N}:=|I(G:Z)_N|. 
%\] 
\begin{lemma}\label{lemma3-1}
For any $N\leq Z$, $\sum_{X\leq N}|I(G:Z)_{X}|=|I(G/N : Z/N)_{\{1\}}|$. 
\end{lemma}
\begin{proof}
Let $\pi_N:Z\rightarrow Z/N$ be the canonical projection.   
We note that  $\{\pi_N(K_i)|i\in I_{G/Z}\} $ is a complete set of representatives of $Z/N$-orbits in $\Conj(G/N)$. 
One also has  
\[
{\rm Ann}_{Z/N}(\pi_N(K_i))=\pi_N({\rm Ann}_{Z}(K_i)N), \quad i\in I_{G/N}.
\] 
In particular, ${\rm Ann}_{Z/N}(\pi_N(K_i))=\{1\}$ if and only if ${\rm Ann}_{Z}(K_i)N=N$ and this is equivalent to ${\rm Ann}_{Z}(K_i)\leq N$. 
This implies that 
\begin{align}\label{eqn3-7}
I(G/N,Z/N)_{\{1\}}=\bigsqcup_{X\leq N}I(G,Z)_X.
\end{align}
Hence we have the lemma. 
\end{proof}

We set 
\begin{align}\label{basis37}
%J_\phi(G):&=\{j\in J(G)| 1\leq N\leq \Ker \phi, j\in J(G)_{N}\},\\
\CB(G)_{\phi}:&=\{e_\phi[K_j]|X\leq \Ker\phi, j\in I(G:Z)_{X}\} 
\end{align}
for any $\phi\in \Irr(Z)$. 

%In the case $\Ker\phi=\{1\}$ we have $J_{\phi}(G)=J(G:Z)_{\{1\}}$. 
\begin{lemma}\label{lemma2-2}
Let $\phi\in\Irr(Z)$.  
%Suppose $\Ker\phi=\{1\}$. 
Then the following hold. 
\begin{enumerate}
%\item[(1)] $J_\phi(G)=J(G:Z)_{\{1\}}$. 

\item[(1)] For $K\in \Conj(G)$, if ${\rm Ann}_Z(K)$ is not a subgroup of $\Ker\phi$, then $e_{\phi}[K]=0$. 

\item[(2)] The set $\CB(G)_{\phi}$ is a basis of $e_\phi Z(\C G)$ and 
\[
\dim e_\phi Z(\C G)=\sum_{X\leq \Ker\phi}|I(G:Z)_{X}|=|I(G/\Ker\phi : Z/\Ker\phi)_{\{1\}}|.
\]  
\end{enumerate}
\end{lemma}
 \begin{proof}
%(1) This is clear form the definition. 
(1) Let $z\in {\rm Ann}_Z(K)-\Ker \phi$. 
Then for any $g\in K$, there exists $h\in G$ such that $g^h=zg$.
Thus $e_\phi[g^h]=\phi(z)e_\phi[g]$.
Therefore, $e_\phi[K]=\phi(z)e_\phi[K]$.
Since $\phi(z)\neq 1$, we have $e_\phi[K]=0$. 

(2)
It follows from \eqref{conj2-1-1} %and the identity $e_{\phi}[zK_j]=\phi(z)e_{\phi}[K_j]$ 
that $\{e_{\phi}[K_j]|j\in I_{G/Z}\}$ spans $e_\phi Z(\C G)$. 
By (1), $e_{\phi}[K_j]=0$ if $j\in I(G:Z)_{N}$ for $N\leq Z$ which is not a subgroup of $\Ker \phi$.
Hence $\CB(G)_{\phi}$ spans $e_\phi Z(\C G)$.  

On the other hand, we note that 
\[
\CB_{\rm conj}(G)=\bigsqcup_{N\leq Z}\bigsqcup_{j\in I(G:Z)_N} \{[z K_j]|z N\in Z/N\}
\] 
is linearly independent.
%Let $N$ be a subgroup of $\Ker\phi$. 
%One also has  
%\[
%|\{[z K_j]|j\in J(G:Z)_N, z N\in Z/N\}|=|Z/N|j(G:Z)_{N}.
%\]
We also consider the set 
\[
\CB(G:Z):=\bigsqcup_{N\leq Z}\bigsqcup_{j\in I(G:Z)_N}\{e_\phi[K_j]|\phi\in \Irr(Z/N)\}. 
\]
Then we have 
\begin{align}
\begin{split}\label{eqn3-17}
\CB(G:Z)&=\bigsqcup_{N\leq Z}\bigsqcup_{\phi\in \Irr(Z/N)}\{e_\phi[K_j]| j\in I(G:Z)_N\}\\
&=\bigsqcup_{\phi\in \Irr(Z)}\bigsqcup_{N\leq \Ker\phi}\{e_\phi[K_j]| j\in I(G:Z)_N\}\\
&=\bigsqcup_{\phi\in \Irr(Z)}\CB(G)_\phi,  
\end{split}
\end{align}
where for any subgroup $N\leq Z$, we identify 
\[
\Irr(Z/N)=\{\phi\in \Irr(Z)|N\leq \Ker \phi\}. 
\] 
As in \eqref{index2-1}, we see that  
\begin{align*}
[\{e_\phi[K_j]|\phi\in \Irr(Z/N)\}:\{[zK_j]|zN\in Z/N\}]^2=|Z/N|^{-|Z/N|} 
\end{align*}
for $ N\leq Z$ and $j\in I(G:Z)_N$. 
Hence 
\begin{align}\label{eqn3-16}
\left[\CB(G:Z):\CB_{\rm conj}(G)\right]^2=\prod_{N\leq Z} |Z/N|^{-|Z/N||I(G:Z)_N|}. 
\end{align}
Therefore, $\CB(G:Z)$ is linearly independent and so is $\CB(G)_\phi$. 
Hence it is a basis of $e_\phi Z(\C G)$. 
In particular, by Lemma \ref{lemma3-1}, 
\begin{align*}
\dim e_\phi Z(\C G)=|\CB(G)_{\phi}|=\sum_{X\leq \Ker\phi}|I(G:Z)_{X}|=|I(G/\Ker\phi:Z/\Ker\phi)_{\{1\}}|. 
\end{align*} 
\end{proof}

For $\phi\in \Irr(Z)$, we define 
\begin{align*}
\gamma_\phi(G):=\gamma(\CB(G)_{\phi}) . 
\end{align*}
%and call it the {\em discriminant} of $G$ for $\phi$. 
We also consider the integer  
\[
c_\phi(G):=\prod_{N\leq\Ker\phi}\prod_{j\in I(G:Z)_{N}} |C_G(g_j)||Z/N|^{-1}.
\]
and define 
\begin{align*}
\mu_\phi(G):=\frac{\gamma_\phi(G)}{c_\phi(G)}.  
\end{align*}

Let $\1$ is the trivial character of the unit group $\{1\}$. 
We note that $I(G:\{1\})_{\{1\}}$ becomes an index set of $\Conj(G)$ and hence $|I(G:\{1\})_{\{1\}}|=|\Conj(G)|$. 
We also see that $\gamma_{\1}(G)=\gamma(G)$ and $c_{\1}(G)=c(G)$.
Hence 
\[
\mu_{\1}(G)=\mu(G)=h(G)^2.
\]  
Therefore, $\mu_\phi(G)$ is a generalization of the square of Harada\rq{}s number. 

%\begin{remark}
%As an index set $J(G:Z)$ of the $Z$-set $\Conj(G)$, we can take an index set $J(G/Z)$ of $\Conj(G/Z)$, that is, 
%\[
%\Conj(G/Z)=\{\ol{K}_j|j\in J(G/Z)\}. 
%\]
%Actually, let $\pi_Z$ be the canonical projection $G\rightarrow G/Z$. 
%Then we take one element $g_i\in \pi_Z^{-1}(\ol{K}_i)$ and set $K_i=g_i^G$ for each $i\in J(G/Z)$. 
%We see that $\pi_Z^{-1}(K_i)$ is the union of all $Z$-orbits of $K_i\in \Conj(G)$ and hence $\{K_i|i\in J(G/Z)\}$ is a complete set of representatives of $Z$-orbits in $\Conj(G)$.  
%%Hence we can take $J(G:Z)=J(G/Z)$. 
%\end{remark}

\begin{remark}
Let $\phi\in \Irr(Z)$. 
Since $\gamma_\phi(G)$ is an absolute value of a determinant and $|\phi(z)|=1$ for any $z\in Z$, it is independent of a choice of a complete set $\{K_i|i\in I(G/Z)\}$ of  representatives of $Z$-orbits in $\Conj(G)$.
% Actually, the index with basis changing matrix to another choice is of the absolute value of the determinant $1$.  
\end{remark}

\subsection{A factorization of $\mu(G)$}
Let $G$ be a finite group and $Z$ a subgroup of $Z(G)$.
It follows from \eqref{eqn3-16} that 
%\begin{align}\label{index3-1}
%[\CB(G:Z):\CB_{\rm conj}(G)]^2=\prod_{N\leq Z}|Z/N|^{-|Z/N|j(G:Z)_N}.  
%\end{align}
%and 
\begin{align}
\begin{split}\label{index3-5}
\gamma(G)&=\gamma(\CB_{\rm conj}(G))\\
&=\gamma(\CB(G:Z))[\CB_{\rm conj}(G):\CB(G:Z)]^2\\
&=\gamma(\CB(G:Z))\prod_{N\leq Z}|Z/N|^{|Z/N||I(G:Z)_N|}. 
\end{split}
\end{align}
For distinct $\phi,\psi\in \Irr(Z)$, $\langle \CB(G)_\phi|\CB(G)_\psi\rangle=0$ since $e_{\phi}e_{\psi}=0$.  
Therefore, by \eqref{eqn3-17}, we obtain 
\begin{align}\label{gram3-10}
\gamma(\CB(G:Z))=\prod_{\phi\in \Irr(Z)}\gamma_\phi(G). 
\end{align}
Therefore, by \eqref{index3-5} we have 
\begin{align}\label{gram3-11}
\gamma(G)=\left(\prod_{\phi\in \Irr(Z)}\gamma_\phi(G)\right)\left(\prod_{N\leq Z}|Z/N|^{|Z/N||I(G:Z)_N|}\right). 
\end{align}
On the other hand, we can deduce 
\begin{align*}
&\left(\prod_{\phi\in \Irr(Z)}c_\phi(G)\right)\left(\prod_{N\leq Z}|Z/N|^{|Z/N||I(G:Z)_N|}\right)\\
&=\left(\prod_{\phi\in \Irr(Z)}\prod_{N\leq \Ker\phi}\prod_{i\in I(G:Z)_{N}}
|C_G(g_i)||Z/N|^{-1}\right)\left(
\prod_{N\leq Z} |Z/N|^{|I(G:Z)_N|}\right)^{|Z/N|}\\
&=\left(\prod_{N\leq Z}\prod_{\phi\in \Irr(Z/N)}\prod_{i\in I(G:Z)_{N}}
|C_G(g_i)||Z/N|^{-1}\right)\left(
\prod_{N\leq Z} |Z/N|^{|I(G:Z)_N|}\right)^{|Z/N|}\\
&=\prod_{N\leq Z}\left(|Z/N|^{-|I(G:Z)_N|}\prod_{i\in I(G:Z)_{N}}
|C_G(g_i)|\right)^{|Z/N|}\left(
\prod_{N\leq Z} |Z/N|^{|I(G:Z)_N|}\right)^{|Z/N|}\\
&=\prod_{N\leq Z}\left(\prod_{i\in I(G:Z)_{N}}|C_G(g_i)|\right)^{|Z/N|}\\
&=\prod_{N\leq Z}\prod_{\substack{i \in I(G:Z)_{N}\\ zN\in Z/N}}|C_G(zg_i)|\\
&=c(G). 
\end{align*}

%
%\begin{align*}
%&c(G)\left(\prod_{N\leq Z}|Z/N|^{|Z/N|j(G:Z)_N}\right)^{-1}\\
%&=\prod_{N\leq Z}\left(\left(\prod_{\substack{zN\in Z/N \\ i\in J(G:Z)_{N}}}|C_G(zg_i)|\right) |Z/N|^{-|Z/N|j(G:Z)_N}\right)\\
%&=\prod_{N\leq Z}\left(\left(\prod_{ i\in J(G:Z)_{N}}|C_G(g_i)|\right)^{|Z/N|}|Z/N|^{-|Z/N|j(G:Z)_N}\right)\\
%&=\prod_{N\leq Z}\prod_{i\in J(G:Z)_{N}}\left(|C_G(g_i)||Z/N|^{-1}\right)^{|Z/N|}\\
%&=\prod_{N\leq Z}\prod_{\phi\in \Irr(Z/N)}\prod_{i\in J(G:Z)_{N}} |C_G(g_i)||Z/N|^{-1}\\
%&=\prod_{\phi\in \Irr(Z)}\prod_{N\leq \Ker\phi}\prod_{i\in J(G:Z)_{N}}|C_G(g_i)||Z/N|^{-1}\\
%&=\prod_{\phi\in \Irr(Z)}c_\phi(G).
%\end{align*}

Combining this and \eqref{gram3-11}, we have  
\begin{theorem}\label{thm3-7}
For any subgroup $Z$ of $Z(G)$,  
\begin{align}\label{gram3-12}
\mu(G)=\prod_{\phi\in \Irr(Z)}\mu_\phi(G). 
\end{align}

\end{theorem}

\subsection{$\mu_\phi(G)$ in the case $\phi$ is non-faithful}
Let $G$ be a finite group and $Z$ a subgroup of $Z(G)$. 
Let $\phi\in \Irr(Z)$.
Then $\phi$ induces a faithful irreducible character $\ol{\phi}$ of $Z/\Ker\phi$ defined by $\ol{\phi}(z\Ker\phi)=\phi(z)$ for $z\in Z$. 
Recall the canonical projection $\pi_N:G\rightarrow G/N$ for any subgroup $N$ of $Z$. 
For the later use, we set 
\[
\kappa_N(G):=\prod_{X\leq N}|X|^{|I(G:Z)_{X}|}
\]
for any $N\leq Z$. 
We first prove the following lemma.   
\begin{lemma}\label{lemma3-6}
Let $\phi\in \Irr(Z)$. 
Then 
\[
\gamma_\phi(G)=\kappa_{\Ker\phi}(G)^2\gamma_{\overline{\phi}}(G/\Ker\phi). 
\]
\end{lemma} 
\begin{proof}
Set $N=\Ker\phi$.
We extend $\pi_N$ to an algebra homomorphism 
\[
\pi_N: \C G \rightarrow \C G/N,\quad  [g]\mapsto [\pi_N(g)],\quad  g\in G.  
\]  
We then have 
\begin{align*}
\pi_N(e_\phi[u])&=\frac{1}{|Z|}\sum_{z\in Z}\phi(z)^{-1}[\pi_N(zu)]\\
%&=\frac{1}{|Z/N|}\sum_{\lambda\in Z/N}\overline{\phi}(\lambda)^{-1}\frac{1}{|N|}\sum_{n\in N}([\lambda\pi_N(nu)])\\
&=\frac{|N|}{|Z|}\sum_{\lambda\in Z/N}\overline{\phi}(\lambda)^{-1}[\lambda\pi_N(u)]\\
&=e_{\overline{\phi}}[\pi_N(u)].
\end{align*}
Thus we have an algebra homomorphism  
\[
F_\phi: e_\phi \C G \rightarrow e_{\overline{\phi}}\C G/N,\quad  e_\phi[g]\mapsto e_{\overline{\phi}}[\pi_N(g)],\quad  g\in G
\]
as the restriction of $\pi_N$ to $e_\phi \C G$. 
In fact, $F_\phi$ is an isomorphism since $F_\phi$ is surjective and 
\begin{align*}
\dim e_\phi \C G=[G:Z]=[G/N:Z/N]=\dim e_{\overline{\phi}} \C G/N.
\end{align*}
%by Lemma \ref{lemma2-2}. 
Let $X\leq N$ and $i\in I(G:Z)_{X}$.
Then  
\begin{align*}
F_\phi(e_\phi[K_i])&=\sum_{g\in K_i}F_\phi(e_\phi[g])\\
&=\sum_{g\in K_i}e_{\overline{\phi}}[\pi_N(g)]\\
&=|X|\sum_{\overline{g}\in \pi_N(K_i)}e_{\overline{\phi}}[\overline{g}]\\
&=|X|e_{\overline{\phi}}[\pi_N(K_i)]. 
\end{align*} 
This shows that $F_\phi(\CB(G)_{\phi})$ is a basis of $e_{\overline{\phi}}Z(\C G/N)$ and 
\[
[F_\phi(\CB(G)_{\phi}):\CB({G}/N)_{\overline{\phi}}]=\prod_{X\leq N}|X|^{|I(G:Z)_{X}|}=\kappa_{N}(G). 
\] 
Therefore, 
\begin{align*}
\gamma_{\phi}(G)=\gamma(\CB(G)_\phi)&=\gamma(F_\phi(\CB(G)_\phi))\\
&=\gamma(\CB(G/N)_{\overline{\phi}})[F_\phi(\CB(G)_{\phi}):\CB({G}/Z)_{\overline{\phi}}]^2\\
&=\kappa_{N}(G)^2 \gamma_{\overline{\phi}}({G}/N). 
\end{align*}
\end{proof}

We next consider the numbers $c_\phi(G)$ and $c_{\overline{\phi}}(G/N)$ for $\phi\in \Irr(Z)$ with $N=\Ker\phi$.
Let $X\leq N$ and $i\in I(G:Z)_{X}$. 
Since $|K_i|=|X||\pi_{N}(K_i)|$, we have 
\begin{align*}
|C_G(g_i)|=\frac{|G|}{|K_i|}=\frac{|G/N||N|}{|X||\pi_{N}(K_i)|}=|N/X||C_{G/N}(\pi_{N}(g_i))|. 
\end{align*}
Since 
\[
I(G/N,Z/N)_{\{1\}}=\bigsqcup_{X\leq N}I(G:Z)_X
\]
by \eqref{eqn3-7}, 
we  deduce that    
\begin{align*}
c_\phi(G)&=\prod_{X\leq N}\prod_{i\in I(G:Z)_{X}}|Z/X|^{-1}|C_{G}({g}_i)|\\
&=\prod_{X\leq N}\prod_{i\in I(G:Z)_{X}}|N/X||Z/X|^{-1}|C_{G/N}(\pi_{N}(g_i))|\\
&=\prod_{i\in I(G/N:Z/N)_{\{1\}}}|Z/N|^{-1}|C_{G/N}(\pi_{N}(g_i))|\\
%&=\prod_{(X,i)\in J(G/N,Z/N)_{\{1\}}}|Z/N|^{-1}|C_{G/N}(\pi_{N}(g_i))|\\
&=c_{\overline{\phi}}(G/N),   
\end{align*}
where we note that $\ol{\phi}$ is faithful. 
By Lemma \ref{lemma3-6} we have 
\begin{align*}
\frac{\gamma_\phi(G)}{c_\phi(G)}&=\kappa_{\Ker\phi}(G)^2\frac{\gamma_{\ol{\phi}}(G/\Ker\phi)}{c_{\ol{\phi}}(G/\Ker\phi)}. 
\end{align*}
Now we have the following proposition. 
\begin{proposition}\label{prop4-12}
Let $G$ be a finite group and $Z$ a subgroup of $Z(G)$. 
Then for any $\phi\in \Irr(Z)$,  
\begin{align}\label{gram3-18}
\mu_\phi(G)=\kappa_{\Ker\phi}(G)^2\mu_{\overline{\phi}}(G/\Ker\phi). 
\end{align}
\end{proposition}
Let $\1_Z$ be a character of $Z$ defined by $\1_Z(z)=1$ for any $z\in Z$. 
Since $\ol{\1_Z}=\1$ and $\mu(G)=\mu_\1(G)$, as a corollary of Proposition  \ref{prop4-12}, we have   
\begin{corollary}\label{coro4-13}
Let $G$ be a finite group and $Z$ a subgroup of $Z(G)$. 
Then 
\begin{align}\label{gram3-19}
\mu_{\1_Z}(G)=(\kappa_{Z}(G)h(G/Z))^2. 
\end{align}
\end{corollary}

As a corollary of Theorem \ref{thm3-7} and Proposition \ref{prop4-12}, we have the following theorem.
\begin{theorem}\label{thm4-12}
Let $G$ be a finite group and $Z$ a subgroup of $Z(G)$. 
Then  
\begin{align}\label{gram3-20}
h(G)^2=\left(\prod_{N\leq Z}\kappa_{N}(G)^{|Z/N|}\right)^2\left(\prod_{\phi\in \Irr(Z)}\mu_{\ol{\phi}}(G/\Ker\phi)\right). 
\end{align}
\end{theorem}

Finally, we consider the case $Z$ is of prime order. 
Let $p$ be a prime dividing $|G|$ and $Z$ a subgroup of $Z(G)$ of order $p$.
Then we have 
\[
\kappa_Z(G)=\prod_{X\leq Z}|X|^{|I(G:Z)_{X}|}=p^{|I(G:Z)_{Z}|}. 
\]
We also see that $\Irr(Z)-\{\1_Z\}=\{\phi^a|a\in \Z_p^\times\}$ for some $\phi\in\Irr(Z)$. 
%Hence by Lemma\ref{Lem3-10},  
%\[
%\prod_{\phi\in\Irr(Z),\Ker\phi=\{1\}}\mu_\phi(G)=\mu_\phi(G)^{p-1}. 
%\]
Therefore, we have 
\begin{corollary}\label{Prime}
Assume that $Z$ is of prime order $p$. 
Take $\phi\in \Irr(Z)-\{\1_Z\}$. 
Then 
\begin{align}\label{eqn3-22}
\mu(G)=p^{2|I(G:Z)_{Z}|}\mu(G/Z)\left(\prod_{a\in \Z_p^\times}\mu_{\phi^a}(G)\right). 
\end{align}
If $p>2$ then 
\[
h(G)=p^{|I(G:Z)_{Z}|}h(G/Z)\left(\prod_{a=1}^{(p-1)/2}\mu_{\phi^a}(G)\right). 
\]
\end{corollary}
\begin{proof}
The second assertion is verified by the fact that $\gamma_\phi(G)=\gamma_{\phi^{-1}}(G)$. 
\end{proof}
\begin{remark}
By Corollary \ref{Prime}, if all of the conditions (i) $p||Z(G)|$ for some prime $p>2$, (ii) $\mu_{\phi}(G)\in\Z$ and (iii) $h(G/Z)\in \Z$ are valid,  then $h(G)\in \Z$. 
We shall see that (ii) does not hold generally in Section 4.  
\end{remark}

\section{Examples}
In this section, we recalculate Harada\rq{}s numbers for some known examples by calculating $\mu_\phi(G)$ for some characters $\phi$. 
\subsection{Abelian groups}
It is clear that Harada\rq{}s number of an abelian group $A$ is $1$. 
We shall prove that $\mu_\phi(A)$ is also $1$ for any subgroup $Z$ of $A$ and $\phi\in \Irr(Z)$. 
\begin{proposition}\label{abel}
Let $A$ be a finite abelian group and $Z$ its subgroup. 
For any $\phi\in \Irr(Z)$, $\mu_\phi(G)=1$. 
\end{proposition}
\begin{proof}
We take a section $s:A/Z\rightarrow A$. 
For simplicity, assign $s(\pi_Z(1))=1$.  
We may set $I_{A/Z}=A/Z$ as  
\[
\Conj(A/Z)=\{\{u\}|u\in A/Z\}.
\]
Since $A$ is abelian, we have $I(A:Z)_{\{1\}}=I_{A/Z}=A/Z$ and $I(A:Z)_{N}=\emptyset$ for any $\{1\}\neq N\leq Z$. 
In particular, $|I(A:Z)_{\{1\}}|=|A/Z|$ and $|I(A:Z)_{N}|=0$ for $N\neq\{1\}$. 

Let $\alpha:A/Z\times A/Z\rightarrow Z$ be a $2$-cocycle defined by 
\[
\alpha(u,v)=s(u)s(v)s(u v)^{-1}
\] 
for $u,v\in A/Z$. 
For any $\phi\in \Irr(Z)$, we have $\CB(A)_\phi=\{e_\phi[s(u)]|u\in A/Z\}$.
Since $[s(u)][s(v)]=[\alpha(u,v)s(u v)]$ we have 
\[
e_\phi[s(u)]e_{\phi}[s(v)]=\phi(\alpha(u,v))e_\phi[s(u v)]. 
\] 
This implies that $\tr_{e_{\phi}\C A}e_{\phi}[s(u)]=\delta_{u,\pi_Z(1)}|A/Z|$ for $u\in A/Z$.  
Therefore, 
\[
\langle e_\phi[s(u)]|e_{\phi}[s(v)]\rangle=\tr_{e_\phi \C A}e_{\phi}[s(u)^{-1}s(v)]=\delta_{u,v}\frac{\phi(\alpha(u^{-1},v))}{\phi(\alpha(u^{-1},u))}|A/Z|  
\]
for $u,v\in A/Z$. 
Since $|\phi(\alpha(u,v))|=1$ for any $u,v\in A/Z$, we have  
\begin{align}\label{eqn4-111}
\gamma_\phi(A)=|A/Z|^{|A/Z|}.
\end{align}
We also see that 
\[
c_\phi(A)=\prod_{N\leq \Ker\phi}\prod_{j\in I(G:Z)_N}|Z/N|^{-1}|A|=(|Z|^{-1}|A|)^{|I(A;Z)_{\{1\}}|}=|A/Z|^{|A/Z|}.
\]
Hence we have $\mu_\phi(A)=1$. 
\end{proof}
 
\subsection{Dihedral groups}
In this section we consider the dihedral groups $G={\Dih}_{2n}$ of order $2n$ for $n\geq 3$.
%Harada\rq{}s conjecture II can be verified from the definition of Harada\rq{}s number directly, that is the complete lists of conjugacy classes and irreducible characters. 
%We shall calculate $\gamma_\phi(G)$ explicitly by using only the list of conjugacy classes.    
We use the following presentation of the dihedral group  
\[
G={\Dih}_{2n}=\langle \sigma, \tau|\sigma^{n}=\tau^2=1, \quad \sigma^\tau=\sigma^{-1}\rangle  
\] 
for $n\in \Zplus$. 

\begin{proposition}\label{prop4-222}
Let $n$ be a positive integer and $G={\Dih}_{2n}$.% the dihedral group of order $2n$. 
Then 
\[
h(G)
=\begin{cases}
n=|G\rq{}|&\text{ if $n$ is odd},\\
\frac{n^2}{4}=|G\rq{}|^2&\text{ if $n$ is even}.
\end{cases}
\]
\end{proposition}
\begin{proof}
We give a proof by dividing three cases. 

{\bf The case $n$ is odd}:
In this case, the conjugacy classes of $G$ is 
\begin{align*}
K_0&=\{e\},\quad K_{i}=\{\sigma^{i},\sigma^{-i}\},\quad i=1,\cdots,\tfrac{n-1}{2},\\
K_{n+1}&=\{\sigma^i\tau|i=0,\cdots, n-1\}. 
\end{align*}
Set $I=\{0,1,\cdots, \tfrac{n-1}{2}, n+1\}$. 
We then obtain the following identities:  
\begin{align}
\begin{split}\label{eqn4-11}
[K_0][K_i]&=[K_i], \quad i\in I,\\ 
[K_i][K_j]&=\begin{cases}
[K_{i+j}]+[K_{|i-j|}]&\text{if }1\leq i\leq j\leq \tfrac{n-1}{2},\quad i\neq j,\\
[K_{2i}]+2[K_{0}]&\text{if }1\leq 2i \leq \frac{n-1}{2},\quad i=j,\\
[K_{n-(i+j)}]+[K_{|i-j|}]&\text{if } \tfrac{n+1}{2}\leq i+j\leq n-1,\quad i\neq j ,\\
[K_{n-2i}]+2[K_{0}]&\text{if }
\tfrac{n+1}{2}\leq 2i \leq n-1,\quad i=j, 
\end{cases} \\
[K_{n+1}][K_{i}]&=\begin{cases}
[K_{n+1}]\quad &i=0,\\
2[K_{n+1}]\quad &i=1,\cdots, \tfrac{n-1}{2},\\
n(\sum_{i=0}^{\frac{n-1}{2}} [K_{i}]) &i=n+1.  
\end{cases}
\end{split}
\end{align}
These identities imply that for any $1\leq i\leq \frac{n-1}{2}$, there exists $1\leq j\leq \frac{n-1}{2}$ uniquely such that $[K_i][K_j]=[K_j]$.
Hence we have  
\begin{align*}
\tr_{Z(\C G)}[K_0]&=\frac{n+3}{2},\quad \tr_{Z(\C G)}[K_i]=3\quad \text{for $i\in I-\{0,n+1\}$}, \\
\tr_{Z(\C G)}[K_{n+1}]&=0. 
\end{align*}
Thus we have 
\begin{align}\label{eqn4-12}
\Gamma(G)=\begin{pmatrix}
\frac{n+3}{2}&3&3&\cdots &3&0\\
3&n+6&6&\cdots &6&0\\
3&6&n+6&\cdots &6&0\\
\vdots& &&\ddots &\vdots&\vdots\\
3&6&6&\cdots &n+6&0\\
0&0&0&\cdots &0&2n^2
\end{pmatrix}.
\end{align}
We can calculate that $\gamma(G)=4n^{\frac{n+5}{2}}$. 
On the other hand, $c(G)=(2n)n^{\frac{n-1}{2}}\cdot 2=4n^{\frac{n+1}{2}}$.
Hence 
\begin{align}\label{eqn4-22}
h(G)^2=\mu(G)=\gamma(G)/c(G)=n^2.
\end{align}
Thus $h(G)=n=|G\rq{}|$. 
%In particular, $h\rq{}(G)=1$ since $|G\rq{}|=n$. 

{\bf The case $n\equiv 2\pmod{4} $}: 
Suppose $n=2q\rq{}$ with {\em odd} integer $q\rq{}\in \Zplus$. 
Then 
%we have a subgroup  $H=\langle \sigma^2, \tau\rangle\cong {\Dih}_{2q\rq{}}$. Since $q\rq{}$ is odd, 
we have $G\cong\langle z\rangle\times {\Dih}_{2q\rq{}}$. 
By \eqref{eqn4-22} and Proposition \ref{prop3-3}, 
\[
h(G)=h({\Dih}_{2q\rq{}})^2=(q\rq{})^2.
\] 
Since $|G'|=|\langle \sigma^2\rangle|=q\rq{}$, we have 
\begin{align}\label{eqn4-15}
h(G)=(q\rq{})^2=|G'|^2. 
\end{align}
%It is known that $k({\Dih}_{4q\rq{}})=1^2\cdot 2^{q\rq{}-1}\cdot (q\rq{})^2$ and $d({\Dih}_{4q\rq{}})=1^4\cdot 2^{q\rq{}-1}$. 
%Thus $h({\Dih}_{4q\rq{}})=(q\rq{})^2$. 
%This will be also observed for the other dihedral groups in Example \ref{GHC-e2}. 

{\bf The case $n\equiv 0\pmod{4}$}: 
Suppose $n=4q$ with integer $q\geq 2$.   
%\[
%G={\Dih}_{4q}=\langle \sigma, \tau|\sigma^{2q}=\tau^2=1, \quad \sigma^\tau=\sigma^{-1}\rangle  
%\] 
%be the dihedral group of order $4q$ for $q\in 2\Zplus$. 
%We take generators $\sigma, \tau $ of $G$ subject to 
%\[
%\sigma^{2q}=1,\quad \tau^2=1, \quad \tau^{-1} \sigma\tau=\sigma^{-1}. 
%\] 
We see that $z=\sigma^{2q}$ generates the center $Z(G)$ and that
\begin{align}
\begin{split}\label{eqn3-71}
K_0&=\{1\},\quad K_i=\{\sigma^i,\sigma^{-i}\},\quad (i=1,\cdots, 2q-1),\quad K_{2q}=\{z\},\\
K_{2q+1}&=\{\sigma^{2i}\tau|i=0,1,\cdots, 2q-1\},\quad K_{2q+2}=\{\sigma^{2i+1}\tau|i=0,1,\cdots, 2q-1\}
\end{split}
\end{align}
are the list of all conjugacy classes of $G$. 

Set $Z=Z(G)=\langle z\rangle$. 
Note that $G/Z \cong {\Dih}_{4q}$.
%Recall the canonical projection $\pi_Z:G\rightarrow G/Z$. 
%Set $\overline{\sigma}=\pi(\sigma)$ and $\overline{\tau}=\pi(\tau)$.  
Among of the conjugacy classes of $G$ in \eqref{eqn3-71}, only $K_{q}, K_{2q+1},K_{2q+2}$ satisfy the condition $zK_i=K_i$. 
Thus 
%\[
%\Conj(G/Z)=\{\pi_Z(K_i)|0\leq i\leq q, 2q+1,2q+2\}
%\]
%and 
\begin{align}
%\Conj(G)&=\{z^a K_i, K_j|0\leq i\leq q/2-1, a=0,1, j=q/2, q+1,q+2\},\\
\Conj(G)_{\{1\}}&=\{z^aK_i|0\leq i\leq q-1, a=0,1\},\\
\Conj(G)_Z&=\{K_{q}, K_{2q+1}, K_{2q+2}\}. 
\end{align}
We may set $I(G:Z)_{\{1\}}=\{0,1,\cdots,q-1\}$ and $I(G:Z)_{Z}=\{q,2q+1,2q+2\}$. 

Let $\phi\in\Irr(Z)$ be the character defined by 
\[
\phi(z)=-1.
\] 
We can show that 
\[
e_\phi [K_i]e_\phi[K_j]=\begin{cases}e_\phi[K_{i+j}]+e_\phi [K_{|i-j|}],&\text{ if }i\neq j,\\
e_\phi [K_{2i}]+2e_\phi[K_0],&\text{ if }i=j
\end{cases}
\]
for $i, j\in I(G:Z)_{\{1\}}-\{0\}$ and that $e_\phi [K_0]e_{\phi}[K_i]=e_{\phi}[K_i]$ for $i\in I(G:Z)_{\{1\}}$. 
We note that $e_\phi[K_{i}]=-e_\phi[K_{2q-i}]$ for $0\leq i\leq q$.
Thus for $1\leq i,j\leq q-1$,  
\begin{align*}
e_\phi[K_i]e_\phi [K_j]=
\begin{cases}
e_{\phi}[K_{i+j}]+e_{\phi}[K_{|i-j|}] &\text{if } 0< i+j<q, i\neq j,\\
e_{\phi}[K_{2j}]+2e_{\phi}[K_0]  &\text{if } 0< 2i<q, i=j,\\
-e_{\phi}[K_{2q-(i+j)}]+e_{\phi}[K_{|i-j|}]  &\text{if } q\leq i+j<2q, i\neq j,\\
-e_{\phi}[K_{2q-2i}]+2e_{\phi}[K_0]  &\text{if } q\leq 2i<2q, i=j.\\
\end{cases}
\end{align*} 
These relations lead that $\tr_{e_{\phi}Z(\C G)}(e_{\phi}[K_0])=q$ and $\tr_{e_\phi Z(\C G)}(e_{\phi}[K_i])=0$ for $1\leq i\leq q-1$. 
Since $0\leq i+j<2q$, $0\leq |i-j|<q$ for any $0\leq i,j\leq q-1$ and ${K_{i^*}}=K_i$,  
\[
\langle e_{\phi}[K_i]|e_{\phi}[K_j] \rangle =\begin{cases} 
q &\text{ if } i=j=0,\\
2q &\text{ if } i=j\neq 0,\\
0 &\text{ if } i\neq j
\end{cases}
\]
for $0\leq i,j\leq q-1$. 
Hence $\gamma_\phi(G)=q(2q)^{q-1}=2^{q-1}q^{q}$. 
On the other hand, 
\[
c_{\phi}(G)=\prod_{i=0}^{q-1}|Z|^{-1}|C_{G}(\sigma^i)|=(4q)(2q)^{q-1}=2^{q+1}q^{q}. 
\] 
Therefore, 
\begin{align}\label{eqn4-13}
\mu_\phi(G)=\frac{1}{4}. 
\end{align}

On the other hand, %since $|I(G:Z)_Z|=3$, 
we have $\kappa_Z(G)=|Z|^{|I(G:Z)_{Z}|}=8$.
Since $G/Z\cong {\Dih}_{4q}$ for $q\geq 2$, by Corollary \ref{coro4-13},  
\[
\mu_{\1_{Z}}(G)=(8h({\Dih}_{4q}))^2. 
\]
By \eqref{eqn4-13}, we have $h(G)^2=\mu(G)=\mu_{\1_{Z}}(G)\mu_{\phi}(G)=(4h({\Dih}_{4q}))^2$ 
and hence 
\[
h({\Dih}_{8q})=4 h({\Dih}_{4q}).
\] 
Since $|{\Dih}_{8q}\rq{}|=q=2|{\Dih}_{4q}\rq{}|$, we have 
\begin{align}
\frac{h({\Dih}_{8q})}{|{\Dih}_{8q}\rq{}|^2}=\frac{h({\Dih}_{4q})}{|{\Dih}_{4q}\rq{}|^2}.  
\end{align}
Set $q=2^e q\rq{}\geq 2$ with odd $q\rq{}$ and $e\geq 0$.
Then by \eqref{eqn4-15},  
\begin{align}
\frac{h({\Dih}_{8q})}{|{\Dih}_{8q}\rq{}|^2}=\frac{h({\Dih}_{4q\rq{}})}{|{\Dih}_{4q\rq{}}\rq{}|^2} =1. 
\end{align}
Consequently, for $q\geq 2$,  
\begin{align}\label{eqn4-21}
h({\Dih}_{8q})=|{\Dih}_{8q}\rq{}|^2=4q^2.
\end{align}
In the case $q=1$, $Z=Z({\Dih}_8)$ and ${\Dih}_8/Z\cong \Z_2^2$. 
This case will be considered in Section \ref{SS4.4}, and it follows from Proposition \ref{prop3-23} with $p=s=2$ that $h({\Dih}_8)=4$. 
Hence \eqref{eqn4-21} also holds in the case $q=1$. 
\end{proof}

\subsection{Generalized quaternions}
In this section we consider the generalized quaternions:  
\[
G={\Qu}_{4n}:=\langle \sigma , \tau|\sigma^{2n}=1, \sigma^n=\tau^{2}, \sigma^\tau=\sigma^{-1}\rangle 
\]
for $n>0$. 
The element $z=\sigma^n$ generates the center $Z=Z(G)$ and $G/Z\cong {\Dih}_{2n}$. 
We also find that $G'=\langle \sigma^2\rangle$ is of order $n$. 
We shall prove the following proposition.
\begin{proposition}\label{prop4-34}
For any $n\in \Z$ with $n\geq 2$, $h({\Qu}_{4n})=n^2=|{\Qu}_{4n}'|^2$. 
\end{proposition}
\begin{proof}
{\bf The case $n$ is odd}: 
The list of conjugacy classes are given by 
\begin{align*}
K_0&=\{1\},\quad K_{i}=\{\sigma^{i},\sigma^{2n-i}\}\quad (i=1,\cdots,n-1),\quad K_{n}=\{z\},\\
K_{n+1}&=\{\sigma^{2i}\tau|i=0,\cdots, n-1\},\quad K_{n+2}=\{\sigma^{2i+1}\tau|i=0,\cdots, n-1\}.  
\end{align*}
Since  
\[
z K_{j}=K_{n-j},\text{ for }0\leq j\leq n-1,\quad z K_{n+1}=Z_{n+2}, 
\]
we have $\Conj(G)_{\{1\}}=\Conj(G)$ and $\Conj(G)_{Z}=\emptyset$. 
In particular, we may set $I(G:Z)_{\{1\}}=\{0,2,4,\cdots, n-1, n+1\}$. 

Let $\phi\in \Irr(Z)$ be the character defined by $\phi(z)=-1$. 
By direct calculations,  
\begin{align*}
&e_\phi[K_{2i}]e_\phi [K_{2j}]\\
&=
\begin{cases}
e_\phi[K_{2j}]&\text {if $0=i\leq j\leq \frac{n-1}{2}$},\\
e_\phi[K_{2(i+j)}]+e_\phi[K_{2|i-j|}]&\text {if $0<j\neq i\leq \frac{n-1}{2}$, $2i+2j\leq n-1$},\\
e_\phi[K_{2n-2(i+j)}]+e_\phi[K_{2|i-j|}]&\text {if $0<j\neq i\leq \frac{n-1}{2}$, $2i+2j\geq n+1$},\\
e_\phi[K_{4i}]+2e_\phi[K_{0}]&\text {if $0<i=j\leq \frac{n-1}{2}$, $4i\leq n-1$},\\
e_\phi[K_{2n-4i}]+2e_\phi[K_{0}]&\text {if $0<i=j\leq \frac{n-1}{2}$, $4i\geq n+1$}, 
\end{cases}
\end{align*}
and 
\begin{align*}
e_\phi[K_{0}]e_\phi [K_{n+1}]&=e_\phi[K_{n+1}],\\
e_\phi[K_{2j}]e_\phi [K_{n+1}]&=2e_\phi[K_{n+1}]\quad (j=1,\cdots, \frac{n-1}{2}), \\
e_\phi[K_{n+1}]^2&=n e_\phi[z\langle \sigma^2\rangle]=-n\sum_{i=0}^{\frac{n-1}{2}}e_\phi[K_{2j}]. 
\end{align*}
These structure constants are similar with the ones in \eqref{eqn4-11} but not the same.
By noting $K_{i^*}=K_i$ for $i=0,\cdots, n$ and $K_{(n+1)^*}=K_{n+2}$,  direct calculation gives the same Gramian matrix in \eqref{eqn4-12} and hence  
\[
\gamma_\phi(G)=\gamma({\Dih}_{2n})=\gamma(G/Z),\quad c_\phi(G)=c(G/Z).
\] 
Therefore, we have $\mu_\phi(G)=\mu(G/Z)$. 
On the other hand, since $|I(G:Z)_{Z}|=0$, we have $\mu_{\1_Z}(G)=\mu(G/Z)$.
Therefore, we have $\mu(G)=\mu(G/Z)^2$.
Thus $h(G)=h(G/Z)^2=h({\Dih}_{2n})^2=n^2=|G\rq{}|^2$. 

{\bf The case $n$ is even}: 
%Let $Z$ and $\phi$ be as in Example \ref{ex4-6}. 
In this case, the list of conjugacy classes of $G$ are given by 
\begin{align*}
K_0&=\{1\},\quad K_{i}=\{\sigma^{i},\sigma^{2n-i}\}\quad (i=1,\cdots,n-1),\quad K_{n}=\{z\},\\
K_{n+1}&=\{\sigma^{2i}\tau|i=0,\cdots, n-1\},\quad K_{n+2}=\{\sigma^{2i+1}\tau|i=0,\cdots, n-1\} 
\end{align*}
and the list is the same one in the case $n$ is odd. 
The difference is   
\[
z K_{j}=K_{j},\text{ for $j=\frac{n}{2}, n+1, n+2$}  
\]
and the situation is the same as the case ${\Dih}_{4n}$. 

We set $I(G:Z)_{\{1\}}=\{0,1,2,\cdots, \frac{n}{2}-1\}$ and $I(G:Z)_{Z}=\{\frac{n}{2},n+1,n+2\}$. 
Then 
\begin{align*}
&e_\phi[K_{i}]e_\phi[K_{j}]\\
&=
\begin{cases}
e_\phi[K_{j}]&\text {if $0=i\leq j\leq \frac{n}{2}-1$},\\
e_\phi[K_{i+j}]+e_\phi[K_{|i-j|}]&\text {if $0<j\neq i\leq \frac{n}{2}-1$, $i+j\leq  \frac{n}{2}$},\\
%e_\phi[K_{|i-j|}]&\text {if $0<j\neq i\leq n/2-1$, $i+j=n/2$},\\
-e_\phi[K_{n-(i+j)}]+e_\phi[K_{|i-j|}]&\text {if $0<j\neq i\leq \frac{n}{2}-1$, $i+j> \frac{n}{2}$},\\
e_\phi[K_{2i}]+2e_\phi[K_{0}]&\text {if $0<i=j\leq \frac{n}{2}-1$, $2i\leq \frac{n}{2}$},\\
-e_\phi[K_{n-2i}]+2e_\phi[K_{0}]&\text {if $0<i=j\leq \frac{n}{2}-1$, $2i> \frac{n}{2}$}. 
\end{cases}
\end{align*}
Hence $\tr_{e_\phi Z(\C G)}e_\phi[K_0]=\frac{n}{2}$, $\tr_{e_\phi Z(\C G)}e_\phi[K_i]=0$ for $i=1,\cdots, \frac{n}{2}-1$. 
Therefore, as in the proof of Proposition \ref{prop4-222}, we obtain $\gamma_\phi(G)=\gamma_\phi({\Dih}_{2n})$.
We also easily check that $c_\phi(G)=c_\phi({\Dih}_{2n})$. 
Thus, in this case, $\mu_\phi(G)=\frac{1}{4}$. 
On the other hand, since $|I(G:Z)_{Z}|=3$, we have $\mu_{\1_Z}(G)=2^6\mu({\Dih}_{2n})=2^6\cdot \left(\frac{n^2}{4}\right)^2=4n^4$.
Therefore we have $\mu(G)=\mu_{\1_Z}(G)\mu_\phi(G)=n^4$ and $h(G)=n^2=|G'|^2$. 
\end{proof}

\subsection{Some $p$-groups}\label{SS4.4}
Let $p$ be a prime. 
Here we give calculations of $\mu(G)$ for some $p$-groups $G$. 
Since a $p$-group has a nontrivial central subgroup of order $p$, we can use the results in Section 3. 
First we give the following proposition. 
\begin{proposition}\label{prop3-23}
Let $p$ be a prime, $G$ a finite $p$-group.
Assume that there is a central subgroup $Z$ of order $p$ such that $G/Z$ is abelian. 
Then  
\begin{align}\label{eqn3-24}
h(G)=p^{\frac{(p-1)(2[s]_p-s)}{2}|Z(G)/Z|},
\end{align}
where $|G/Z(G)|=p^s$ and $[s]_p=(p^{s}-1)/(p-1)$.
\end{proposition}
\begin{remark}
This is the case $G$ is nilpotent of class $2$. 
Harada's conjecture II is true in this case due to \cite{Miyamoto}.  
\end{remark}
\begin{proof}
Since $G/Z$ is abelian, we may set $I_{G/Z}=\{1,\cdots, |G/Z|\}$. We then set 
\[
I(G:Z)_{\{1\}}=\{1,\cdots, t\},\quad I(G:Z)_Z=\{t+1,\cdots, |G/Z|\}
\]
and take $g_i\in G$ so that $g_i^G=K_i$ for $i=1,\cdots, |G/Z|$. 
We see that 
\[
K_i=\begin{cases}
\{g_i\}&\text{ for } i\in I(G:Z)_{\{1\}},\\
\{zg_i|z\in Z\}&\text{ for }i\in I(G:Z)_Z. 
\end{cases}
\]
In particular, we have 
\[
\Conj(G:Z)_{\{1\}}=\{\{g\}|g\in Z(G)\}= \Conj(Z(G):Z)_{\{1\}}
\]
and hence we may identify $I(G:Z)_{\{1\}}=I(Z(G):Z)_{\{1\}}$. 
In particular, $t=|Z(G)/Z|$. 

Let $\phi\in \Irr(Z)$ and assume that $\phi\neq \1_Z$.
Since $|Z|=p$, $\phi$ is faithful and 
\[
\CB(G)_\phi=\{e_\phi[g_i]|i\in I(G:Z)_{\{1\}}\}=\{e_\phi[g_i]|i\in I(Z(G):Z)_{\{1\}}\}=\CB(Z(G))_\phi.
\] 
%where we regard $\C Z(G)\subset \C G$. 
Therefore, we have 
\[
\gamma_\phi(G)=\gamma_\phi(Z(G))=|Z(G)/Z|^{|Z(G)/Z|}
\]
by \eqref{eqn4-111}.  
We also have     
\[
c_\phi(G)=\prod_{i\in I(G:Z)_{\{1\}}}(|Z|^{-1}|G|)=|G/Z|^{t}=|G/Z|^{|Z(G)/Z|}.
\]
Thus 
\[
\mu_\phi(G)=|G/Z(G)|^{-|Z(G)/Z|}. 
\]

One also has $\Conj(G:Z)_Z=\{K_i|i=t+1,\cdots, |G/Z|\}$ and 
\[
|I(G:Z)_Z|=(|G|-|Z(G)|)/|Z|=(|G/Z(G)|-1)|Z(G)/Z|.
\] 
Since $\mu(G/Z)=1$, by Corollary \ref{Prime}, 
\begin{align*}
\mu(G)&=|Z|^{2|Z(G)/Z|(|G/Z(G)|-1)}\times \left(|G/Z(G)|^{-|Z(G)/Z|}\right)^{p-1}\\
 &=\left(|Z|^{2(|G/Z(G)|-1)}|G/Z(G)|^{-(p-1)}\right)^{|Z(G)/Z|}\\
 &=\left(p^{2(p^s-1)-s(p-1)}\right)^{|Z(G)/Z|}\\
 &=\left(p^{(2[s]_p-s)(p-1)}\right)^{|Z(G)/Z|}.
\end{align*}
%We set $|G/Z(G)|=p^s$ with $s\geq 0$.
Since $\mu(G)=h(G)^2$ we have the desired formula. 
\end{proof}

\begin{remark}
By Proposition \ref{prop3-23}, 
\begin{align*}
h(G)=\begin{cases}
2^{(2(2^s-1)-s)\frac{|Z(G)/Z|}{2}}&\text{ if }p=2,\\
p^{((p^s-1)-s(p-1)/2)|Z(G)/Z|}&\text{ if $p$ is odd prime}.
\end{cases}
\end{align*}
Therefore, if $p>2$ or $p=2$ and $Z\neq Z(G)$ then $h(G)\in \Z$.  
If $Z(G)=Z$ is of order $2$, then the commutator subgroup $G\rq{}$ is of order $1$ or $2$.
If $G\rq{}=\{1\}$ then $G$ is abelian and $s=0$. 
This implies that $h(G)=1$. 
If $G\rq{}=Z=Z(G)$ then it is an extraspecial $2$-group since Frattini subgroup is $Z$.
Then it follows from \cite[Theorem 4.18]{Suzuki2} that %$G$ is a central product of the dihedral group of order $8$ or the quaternion group of order $8$ (see ). Hence 
$|G|=2^{2k+1}$ for some $k\in \N$ and $s=2k$ is even.
Thus $h(G)$ is an integer.    
\end{remark}

\begin{remark}Let $G$ and $Z$ be as above.
Then $ G'\leq Z$ since $G/Z$ is abelian.   
If $|G'|\neq 1$ then $|G'|=|Z|=p$. 
In this case, since $h(G)$ is a power of $p$ with positive exponent.
Hence $|G'||h(G)$.  

\end{remark}

Let $p$ be a prime and $d$ a positive integer greater than $2$. 
Consider the group
\[
G=M(p^d):=\langle x,y|x^{p^{d-1}}=y^{p}=1, x^y=x^{1+p^{d-2}}\rangle 
\]
(see \cite[Theorem 4.1]{Suzuki2}). 
\begin{proposition}\label{prop-Mp}
$h(M(p^d))=p^{p^{d-2}(p-1)}$.
\end{proposition}
\begin{proof}
We see that $Z(G)=\langle x^{p}\rangle$ is of order $p^{d-2}$ and $Z:=\langle x^{p^{d-2}} \rangle=G\rq{}$ is of order $2$. 
We note that $G/Z$ is abelian, $|G/Z(G)|=p^2$ and that $|Z(G)/Z|=p^{d-3}$.
Hence it follows from Proposition \ref{prop3-23} for $s=2$ that $h(M(p^d))=p^{p^{d-2}(p-1)}$.
\end{proof}

Finally, we consider the semi-dihedral groups ${\SD}_{2^{n}}$ for $n\geq 2$.
The semi-dihedral group ${\rm SD}_{2^{n}}$ has following generators and presentations  
\[
G={\rm SD}_{2^n}:=\langle \sigma, \tau|\sigma^{2^{n-1}}=\tau^2=1,\quad \sigma^{\tau}=\sigma^{2^{n-2}-1}\rangle  
\]
for $n\geq 4$. 
We see that $G'=\langle \sigma^2\rangle$ of order $2^{n-2}$
\begin{proposition}\label{propsd}
For any $n\in \Z$ with $n\geq 2$, $h(\SD_{2^n}) =2^{2(n-2)}=|G'|^2$. 
\end{proposition}
\begin{proof}
We see that $Z=Z(G)$ is generated by $z:=\sigma^{2^{n-2}}$. 
Set $q=2^{n-2}$. Then  
\begin{align*}
K_0&=\{1\},\quad K_{\frac{q}{2}}=\{\sigma^{\frac{q}{2}}, \sigma^{-\frac{q}{2}}\},\quad K_{q}=\{z\} \\
K_{i}&=\{\sigma^{i},(z\sigma)^{2n-i}\} \quad (i=1,\cdots,\frac{q}{2}-1),\\
K_{i}&=zK_{i-\frac{q}{2}} \quad (i=\frac{q}{2}+1,\cdots,q-1),\\
K_{q+1}&=\{\sigma^{2i}\tau|i=0,\cdots, n-1\},\quad K_{q+2}=\{\sigma^{2i+1}\tau|i=0,\cdots, n-1\} 
\end{align*}
is the list of all conjugacy classes of ${\SD}_{2^n}$. 
We note that 
\[
\Conj(G)_{Z}=\{K_{q/2}, K_{q+1}, K_{q+2}\}
\] 
and may set $I(G:Z)_{\{1\}}=\{0,1,2,\cdots, \frac{q}{2}-1\}$ and $I(G:Z)_{Z}=\{\frac{q}{2},q+1,q+2\}$.  
By using fact
\[
(\sigma^{i})^G=z^iK_{q-i},\text{ for }\frac{q}{2}<i<q, 
\]
we obtain 
\begin{align*}
&e_\phi[K_{i}]e_\phi [K_{j}]\\
&=
\begin{cases}
e_\phi[K_{j}]&\text {if $0=i\leq j\leq \frac{q}{2}-1$},\\
e_\phi[K_{i+j}]+(-1)^{i}e_\phi[K_{i-j}]&\text {if $0<j<i\leq \frac{q}{2}$ and $i+j\leq \frac{q}{2}$},\\
(-1)^{i+j}e_\phi[K_{q-(i+j)}]+(-1)^i e_\phi[K_{i-j}]&\text {if $0<j< i\leq \frac{q}{2}-1$ and $i+j> \frac{q}{2}$},\\
e_\phi[K_{2i}]+(-1)^i2e_\phi[K_{0}]&\text {if $0<i=j\leq \frac{q}{2}-1$ and $2i\leq \frac{q}{2}$},\\
e_\phi[K_{q-2i}]+(-1)^i2e_\phi[K_{0}]&\text {if $0<i=j\leq \frac{q}{2}-1$ and $2i> \frac{q}{2}$}. 
\end{cases}
\end{align*}
These identities imply that $\tr_{e_\phi Z(\C G)}e_\phi[K_0]=\frac{q}{2}$ and $\tr_{e_\phi Z(\C G)}e_\phi[K_i]=0$ for $i=1,\cdots, \frac{q}{2}-1$. 
Therefore, we can show that $\gamma_\phi(G)=\frac{q^{\frac{q}{2}}}{2}$.
We also calculate $c_\phi(G)=2q^{\frac{q}{2}}$ and get $\mu_\phi(G)=\frac{1}{4}$. 
On the other hand, since $j(G:Z)_{Z}=3$, we have $\mu_{\1_Z}(G)=2^6\mu(G/Z)$.
Since $\SD_{2^n}/Z\cong {\Dih}_{2^{n-1}}$, we have $\mu(G)=2^4\mu({\Dih}_{2^{n-1}})$ and $h(G)=4h({\Dih}_{2^{n-1}})=4\cdot (2^{n-3})^2=4^{n-2}$. 
\end{proof}
\begin{remark}
As proved in \cite[Theorem 4.1]{Suzuki2}, a non-abelian $p$-group $G$ containing cyclic maximal subgroup is of the form $M(p^n)$ if $p$ is odd. 
If $p=2$, then $G$ is one of $M(2^n)$, the dihedral groups ${\mathrm D}_{2^n}$, generalized quaternion ${\mathrm Q}_{2^n}$ or the semi-dihedral groups ${\mathrm SD}_{2^n}$.
Hence Harada's conjecture II and Harada-Chigira's conjecture are valid for such groups.    
\end{remark}

\subsection{Central products}
%Then $G/Z\cong \prod_{i=1}^{r}A_i/Z$.
Let $G_1$ and $G_2$ be finite group and $H_i\subset Z(G_i)$ a central subgroup of $G_i$ for $i=1,2$. 
Assume that $H_1\cong H_2$ and $\theta:H_1\rightarrow H_2$ is a group isomorphism. 
Then the quotient $G=(G_1\times G_2)/N$ by the subgroup $N=\{(h_1,\theta(h_1^{-1}))|h_1\in H_1\}$ of $H_1\times H_2$ is called a \textit{central product} of $G_1$ and $G_2$ with respect to $H_1$ and $H_2$ (cf. \cite[P29]{Gorenstein}, \cite[P137]{Suzuki1}).  

Let $\pi:G_1\times G_2\rightarrow G$ be the canonical projection.
We regard $G_i$ to be a subgroup of $G_1\times G_2$ in natural way,  and set $Z=\pi(H_1)=\pi(H_2)$. 
We have 
\[
G/Z\cong (G_1\times G_2)/(H_1\times H_2)\cong G_1/H_1\times G_2/H_2.
\]

For $s=1,2$, let $\Conj(G_s/H_s)=\{\ol{K}^{(s)}_i|i\in I_{G_s/H_s}\}$ and take a set of all representatives $\{K_i^{(s)}|i\in I_{G_s/H_s}\}$ of the $Z$-orbits of the $Z$-set $\Conj(G_s)$ and element$g_i^{(s)}\in K^{(s)}$.
We see that $I_{G/Z}=I_{G_1/H_1}\times I_{G_2/H_2}$ and that 
\[
|(g_i^{(1)}g_j^{(2)})^G|=|K_{i}^{(1)}||K_{j}^{(2)}|\quad\text{for } (i,j)\in I(G_1,H_1)_{\{1\}}\times I(G_2,H_2)_{\{1\}}. 
\] 
%We note that 
%\[
%\Conj(G_1\times G_2)=\{K_i\times K_j|i\in I_{G_1/H_1}, j\in I_{G_2/H_2}\}
%\]
%and 
%\[
%{\rm Ann}_{H_1\times H_2}(K_i\times K_j)=\begin{cases}
%\{1\}&\text{ if $i\in I(G:H_1)_{\{1\}}$ and $j\in I(G:H_2)_{\{1\}}$},\\
%H_1&\text{ if $i\in I(G:H_1)_{H_1}$ and $j\in I(G:H_2)_{\{1\}}$},\\
%H_2&\text{ if $i\in I(G:H_1)_{\{1\}}$ and $j\in I(G:H_2)_{H_2}$},\\
%H_1\times H_2&\text{ if $i\in I(G:H_1)_{H_1}$ and $j\in I(G:H_2)_{H_2}$}. \end{cases}
%\]
%Hence for any subgroup $X$ of $H_1\times H_2$, we have 
%\begin{align*}
%I(G_1\times G_2,H_1\times H_2)_X=
%\begin{cases}
%I(G_1,H_1)_{\{1\}}\times I(G_2,H_2)_{\{1\}}&\text{ if $X=\{1\}$},\\
%I(G_1,H_1)_{H_1}\times I(G_2,H_2)_{\{1\}}&\text{ if $X=H_1$},\\
%I(G_1,H_1)_{\{1\}}\times I(G_2,H_2)_{H_2}&\text{ if $X=H_2$},\\
%I(G_1,H_1)_{H_1}\times I(G_2,H_2)_{H_2}&\text{ if $X=H_1\times H_2$},\\
%\emptyset &\text{ otherwise. }\end{cases}
%\end{align*}
%This shows that 
%%\begin{align*}
%%&\kappa_{H_1\times H_2}(G_1\times G_2)\\
%%&=p^{|I(G_1,H_1)_{H_1}||I(G_2,H_2)_{\{1\}}|+|I(G_1,H_1)_{\{1\}}||I(G_2,H_2)_{H_2}|+2|I(G_1,H_1)_{H_1}||I(G_2,H_2)_{H_2}|}\\
%%&=p^{|I_{G_1/H_1}||I(G_2,H_2)_{H_2}|+|I(G_1,H_1)_{H_1}||I_{G_2/H_2}|}. 
%%\end{align*}
%\begin{align*}
%\kappa_{N}(G_1\times G_2)=p^{|I(G_1\times G_2,H_1\times  H_2)_{N}|}=1. 
%\end{align*}
%Hence for any faithful character $\phi$ on $(H_1\times H_2)/N\subset Z(G)$, if one considers the character $\psi$ of $H_1\times H_2$ such that $\psi(u)=\phi(uN)$, then 
%\begin{align*}
%\mu_\psi(G_1\times G_2)=\mu_{\phi}(G). 
%\end{align*}  
%Hence we can identify $I_{G/Z}=I_{G_1/H_1}\times I_{G_2/H_2}$. 

Assume that $Z$ is of prime order $p$. 
Then we see that 
\[
I(G:Z)_{\{1\}}=I(G_1;H_1)_{\{1\}}\times I(G_2:H_2)_{\{1\}}. 
\] 
%and $I(G:Z)_{Z}=I_{G_1/H_1}\times I_{G_2/H_2}-I(G:Z)_{\{1\}}$. 
For any $\phi\in \Irr(Z)$ with $\Ker \phi=\{1\}$, we have  
\[
\CB(G)_\phi=\{uv|u\in \CB(G_1)_{\phi_1}, v\in  \CB(G_2)_{\phi_2}\},
\] 
where $\phi_i$ is defined by $\phi_i(h)=\phi(\pi(h))$ for $h\in H_i$, and we identify $\C[G_i]\subset \C[G]$ for $i=1,2$. 
Since $\C[G_1]$ and $\C[G_2]$ mutually commute, we have 
\[
\gamma_\phi(G)=\gamma_{\phi_1}(G_1)^{|I(G_2,H_2)_{\{1\}}|} \gamma_{\phi_2}(G_2)^{|I(G_1,H_1)_{\{1\}}|}. 
\]
Furthermore, one has 
\[
|C_G(g^{(1)}_ig^{(2)}_j)|=\frac{|G|}{|(g^{(1)}_ig^{(2)}_j)^G|}
=\frac{|G_1||G_2|}{p|K^{(1)}_i||K^{(2)}_j|}
=p^{-1}|C_{G_1}(g^{(1)}_i)||C_{G_2}(g^{(2)}_j)|
\]
for $(i,j)\in I(G_1,H_1)_{\{1\}}\times I(G_2,H_2)_{\{1\}}$. 
Hence 
\begin{align*}
c_\phi(G)&=\prod_{\substack{i\in I(G_1,H_1)_{\{1\}}\\ j\in  I(G_2,H_2)_{\{1\}}}} 
|C_G((g_i,g_j)N)||Z|^{-1}\\
&=\prod_{\substack{i\in I(G_1,H_1)_{\{1\}}\\ j\in  I(G_2,H_2)_{\{1\}}}} 
|C_{G_1}(g_i)||C_{G_2}(g_j))|p^{-2}\\
&=\left(\prod_{i\in I(G_1,H_1)_{\{1\}}} 
|C_{G_1}(g_i)|p^{-1}\right)^{|I(G_2,H_2)_{\{1\}}|}\\
&\quad \times \left(\prod_{j\in I(G_2,H_2)_{\{1\}}}|C_{G_2}(g_j))|p^{-1}\right)^{|I(G_1,H_1)_{\{1\}}|}\\
&=c_{\phi_1}(G_1)^{|I(G_2,H_2)_{\{1\}}|}
c_{\phi_2}(G_2)^{|I(G_1,H_1)_{\{1\}}|}
\end{align*}
and thus 
\begin{align*}
\mu_\phi(G)=\mu_{\phi_1}(G_1)^{|I(G_2,H_2)_{\{1\}}|}\mu_{\phi_2}(G_2)^{|I(G_1,H_1)_{\{1\}}|}. 
\end{align*}

On the other hand, 
%\begin{align*}
%|I(G:Z)_Z|&=|I_{G_1/H_1}||I_{G_2/H_2}|-|I(G_1;H_1)_{\{1\}}||I(G_2;H_2)_{\{1\}}|. 
%\end{align*}
by Proposition \ref{prop3-3},    
\[
\mu(G/Z)=\mu\left(G_1/H_1\times   G_2/H_2 \right)=\mu(G_1/H_1)^{|I_{G_2/H_2}|}\mu(G_2/H_2)^{|I_{G_1/H_1}|}. 
\]
%\begin{align}\label{eqn3-22}
%\mu(G)=p^{2|I(G:Z)_{Z}|}\mu(G/Z)\left(\prod_{a\in \Z_p^\times}\mu_{\phi^a}(G)\right). 
%\end{align}
It follows from \eqref{eqn3-22} that  
\begin{align*}
\mu(G)&=p^{2|I(G:Z)_{Z}|}\mu(G/Z)\left(\prod_{a\in \Z_p^\times}\mu_{\phi^a}(G)\right)\\
&=p^{2|I(G:Z)_{Z}|}\mu(G_1/H_1)^{|I_{G_2/H_2}|}\prod_{a\in \Z_p^\times}\mu_{\phi_1^a}(G_1)^{|I(G_2:H_2)_{\{1\}}|}\\
&\quad \times
\mu(G_2/H_2)^{|I_{G_1/H_1}|}\prod_{a\in \Z_p^\times}\mu_{\phi_2^a}(G_2)^{|I(G_1:H_1)_{\{1\}}|}\\
&=p^{2|I(G:Z)_{Z}|}\\
&\quad \times \mu(G_1/H_1)^{|I_{G_2/H_2}|}\left(\mu(G_1)p^{-2|I(G_1:H_1)_{H_1}|}\mu(G_1/H_1)^{-1} \right)^{|I(G_2:H_2)_{\{1\}}|}\\
&\quad \times
\mu(G_2/H_2)^{|I_{G_1/H_1}|}\left(\mu(G_2)p^{-2|I(G_2:H_2)_{H_2}|}\mu(G_2/H_2)^{-1} \right)^{|I(G_1:H_1)_{\{1\}}|}\\
%&=p^{2|I(G:Z)_{Z}|-2|I(G_1:H_1)_{H_1}||I(G_2:H_2)_{\{1\}}|-2|I(G_1:H_1)_{\{1\}}||I(G_2:H_2)_{H_2}|}\\
%&\quad \times \mu(G_1/H_1)^{|I_{G_2/H_2}|-|I(G_2:H_2)_{\{1\}}|}\mu(G_2/H_2)^{|I_{G_1/H_1}|-|I(G_1:H_1)_{\{1\}}|}\\
&=p^{2|I(G_1:H_1)_{H_1}||I(G_2:H_2)_{H_2}|}\\
&\quad \times \mu(G_1)^{|I(G_2:H_2)_{\{1\}}|}\mu(G_2)^{|I(G_1:H_1)_{\{1\}}|}\\
&\quad \times \mu(G_1/H_1)^{|I(G_2:H_2)_{H_2}|}\mu(G_2/H_2)^{|I(G_1:H_1)_{H_1}|}.
\end{align*}
Consequently, we have 
\begin{align}
\begin{split}
\label{CP003}
h(G)=&p^{|I(G_1:H_1)_{H_1}||I(G_2:H_2)_{H_2}|}\\
&\times h(G_1)^{|I(G_2:H_2)_{\{1\}}|}h(G_2)^{|I(G_1:H_1)_{\{1\}}|}\\
&\times h(G_1/H_1)^{|I(G_2:H_2)_{H_2}|}h(G_2/H_2)^{|I(G_1:H_1)_{H_1}|}. 
\end{split}
\end{align}
Now we have the following theorem. 
\begin{theorem}\label{propcross}
Let $G_1$ and $G_2$ be finite groups containing central subgroups $H_1$ and $H_2$ of prime order $p$, respectively.
Let $G$ be the central product of $G_1$ and $G_2$ with respect to %a group  isomorphism between 
$H_1$ and $H_2$.
Then $h(G)$ is given by \eqref{CP003}. 
In particular, if all $h(G_1)$, $h(G_2)$, $h(G_1/H_1)$ and $h(G_2/H_2)$ are integer so is $h(G)$. 
\end{theorem}
\begin{remark}
In the case $Z=\{1\}$, a central product is a direct product. 
This is the case $p=1$, that is $H_1$ and $H_2$ are trivial, in Proposition \ref{prop3-3}. 
\end{remark}
\begin{example}
As a trivial case, for $i=1,2$, we consider a finite group $G_i=Z\times T_i$, where $Z$ is a group of prime order $p$ and $T_i$ is a group.
Let $G =(G_1\times G_2)/N$ with $H_1=H_2=Z$ and $N=(z,z^{-1})|z\in Z\}$.
Since $G\cong Z\times T_1\times T_2$ we have 
\[
h(G)=h(T_1)^{p|I_{T_2}|}h(T_2)^{p|I_{T_1}|}
\] 
by Proposition \ref{prop3-3}. 

On the other hand, we have $I(G_i, Z)_{(1)}=I_{T_i}$ and $I(G_i, Z)_{Z}=\emptyset$ for $i=1,2$. 
Thus \eqref{CP003} gives 
 \[
 h(G)=h(Z\times T_1)^{|I_{T_2}|}h(Z\times T_2)^{|I_{T_1}|}
 =h(T_1)^{p|I_{T_2}|}h(T_2)^{p|I_{T_1}|}.
 \]
\end{example}
\begin{corollary}
Let $G$, $G_1,G_2,H_1,H_2$ be as in Theorem \ref{propcross}. 
If $|G_i'||h(G_i)$ and $|(G_i/H_i)'||h(G_i/H_i)$ for $i=1,2$, then $|G'||h(G)$.
 \end{corollary}
\begin{proof}
Since $H_i$ is of prime order, $I(G_i,H_i)_{H_i}\neq \emptyset$ if and only if  $H_i\subset (G_i)'$, and then $|G_i'|=p|(G_i/H_i)'|$.
Hence $Z\subset (G_1\times G_2)'$ if one of $I(G_1,H_1)_{H_1}$ and  $I(G_2,H_2)_{H_2}$ are non-empty. 
In the case, 
\[
|G'|=|Z||G'/Z|=p|(G_1/H_1)'||(G_2/H_2)'|.
\]
Moreover, if $I(G_1,H_1)_{H_1}\neq \emptyset$ or $I(G_2,H_2)_{H_2}\neq \emptyset$ then 
\[
|G'|=|G_1'||(G_2/H_2)'|\quad \text{or } 
|G'|=|(G_1/H_1)'||G_2'|, 
\]  
respectively. 
By \eqref{CP003}, we have $|G'||h(G)$.  

In the case $I(G_1,H_1)_{H_1}=I(G_2,H_2)_{H_2}=\emptyset$, then $Z\cap G'=\{1\}$ and 
\[
|G'|=|(G/Z)'|=|(G_1/H_1)'||(G_2/H_2)'|=|G_1'||G_2'|. 
\]
Hence \eqref{CP003} proves $|G'||h(G)$. 
\end{proof}


\begin{thebibliography}{10000}
\bibitem[ATLAS]{ATLAS} 
J.H. Conway, R.T. Curtis, S.P. Norton, R.A. Parker and R.A. Wilson, 
\textit{Atlas of finite groups}, Oxford University Press, Oxford, 1985.

%\bibitem[Abe00]{Abe00} T.~Abe, Fusion rules for the free bosonic orbifold vertex operator algebra, \textit{J.~Algebra} {\bf229} (2000), 333--374.

%\bibitem[Ab1]{Abe01}
%T.~Abe, 
%Fusion rules for the charge conjugation orbifold,
%\textit{J. Algebra}, {\bf  242}, no. 2, 624--655 (2001). 
%
%\bibitem[Ab2]{Abe05}
%T.~Abe, 
%Rationality of the vertex operator algebra $V_L^+$ for a positive definite even lattice $L$,  \textit{Math. Z.}, {\bf 249}, no. 2, 455--484 (2005). 

%\bibitem[ABD]{ABD} T.~Abe, G.~Buhl and C.~Dong, Rationality, Regularity, and  $C_2$-cofiniteness, \textit{Trans.~Amer.~Math.~Soc.}  {\bf 356},  no. 8, 3391--3402  (2004).

%\bibitem[ADL05]{ADL} T.~Abe, C.~Dong and H.~Li, Fusion Rules for the Vertex Operator Algebras $M(1)^+$ and $V_L^+$, \textit{Commun. Math. Phys.} {\bf253} 171--219, (2005).

%\bibitem[ALY]{AbeLamYamada17} 
%T.~Abe, C. H. Lam and H.~Yamada, 
%A remark on $\Z_p$-orbifold constructions of the Moonshine vertex operator algebra,  
%arXiv:1705.09022.
%
%\bibitem[BK]{BakalovKac04} 
%B. Bakalov, V. Kac, 
%Twisted modules over lattice vertex algebras, 
%in {\it Lie theory and its applications in physics V}, ed. H.-D. Doebner and V.K. Dobrev, 
%World Sci. Publishing, River Edge, NJ, 2004, pp. 3-26.
%
%\bibitem[CM]{CarnahanMiyamoto16}
%S. Carnahan and M. Miyamoto, 
%Rationality of fixed-point vertex operator algebras, arXiv:1603.05645v1. 
%
%\bibitem[CL]{ChenLam16}
%H.-Y. Chen and  C.H. Lam, 
%Quantum dimensions and fusion rules of the VOA $V_{L_{\mathcal{C}\times\mathcal{D}}}^\tau$,  
%\textit{J. Algebra}, {\bf 459}, 309--349  (2016).
%
%\bibitem[CLS]{ChenLamShimakura16} 
%H.-Y. Chen,  C. H. Lam and H. Shimakura, 
%On $\Z_3$-orbifold construction of the Moonshine vertex operator algebra,  arXiv:1606.05961 
%
%\bibitem[D]{Dong93} C.~ Dong, 
%Vertex algebras associated with even lattices,  
%{\it J. Algebra} {\bf 161}, no. 1, 245--265 (1993).
%%
%%\bibitem[D2]{Dong94} C.~ Dong, Twisted modules for vertex algebras associated with even lattices, {\it J. Algebra} {\bf 165}, no. 1, 91--112 (1994). 
%
%\bibitem[DJX]{DongJiaoXu13} 
%C. Dong, X. Jiao, F. Xu,
%Quantum dimensions and quantum Galois theory,
%\textit{Trans. Amer. Math. Soc.}, {\bf 365}, no. 12, 6441--6469 (2013).
%
%\bibitem[DL1]{DongLepowsky93} 
%C. Dong and J. Lepowsky, 
%{\it Generalized vertex algebras and relative vertex operators}, 
%Progress in Mathematics, {\bf vol. 112}, Birkhauser Boston Inc., Boston, MA, (1993). 
%
%%\bibitem[DLM96]{DongLiMason96}
%%C.~Dong, H.-S.~Li and G.~Mason, 
%%Simple currents and extensions of vertex operator algebras, 
%%\textit{Commun.~Math. Phys} {\bf 180}, 671--707, (1996).
%
%%\bibitem[DLM1]{DongLiMason98}
%%C.~Dong, H.-S.~Li and G.~Mason, 
%%Twisted representations of vertex operator algebras, 
%%\textit{Math.~Ann.} {\bf 310}, 571--600, (1998).
%
%
%%\bibitem[DLM2]{DLM2} C.~Dong, H.-S.~Li and G.~Mason, Vertex operator algebras and associative algebras,  \textit{J.~Algebra} {\bf 206},  no. 1, 67--96,  (1998).
%
%%\bibitem[DLM2]{DongLiMason00}
%%C.~Dong, H.-S.~Li and G.~Mason, 
%%Modular-invariance of trace functions in orbifold theory and generalized moonshine, 
%%\textit{Commun. Math. Phys.} {\bf 214}, no. 1, 1--56, (2000).
%
%
%
%%\bibitem[DLTYY]{DongLamTanabeYamadaYokoyama04} C.~Dong, C.~Lam, K.~Tanabe, H.~Yamada, K.~Yokoyama, $\Z_3$ symmetry and $W_3$ algebra in lattice vertex operator algebras,
%%{\it Pacific J. Math.}, {\bf  215}, no. 2, 245--296 (2004). 
%%
%%
%%\bibitem[DLY1]{DongLamYamada99}
%%C.~Dong, C.-H.~Lam and H.~Yamada,
%%Decomposition of the Vertex Operator Algebra $V_{\sqrt{2}A_3}$, 
%%\textit{J. Alg.} {\bf 222}, 500--510, (1999).
%%
%%\bibitem[DLY2]{DongLamYamada01}
%%C.~Dong, C.-H.~Lam and H.~Yamada,
%%Decomposition of the Vertex Operator Algebra $V_{\sqrt{2}D_l}$, 
%%\textit{Commun. Contemp. Math.} {\bf 3}, 137--151, (2001).
%%
%%\bibitem[DLY3]{DongLamYamada09}
%%C.~Dong, C.-H.~Lam and H.~Yamada,
%%$W$-algebras related to parafermion algebras, 
%%\textit{J. Alg.} {\bf 322}, 2366--2403, (2009).
%%
%%
%%\bibitem[DLWY]{DongLamWangYamada10}
%%C.~Dong, C.-H.~Lam, Q.~Wang and H.~Yamada,
%%The structure of parafermion vertex operator algebras, 
%%\textit{J. Alg.} {\bf 323}, 371--381, (2010).
%%
%\bibitem[DL2]{DongLepowsky96} C. Dong and J. Lepowsky, 
%The algebraic structure of relative twisted vertex operators, 
%\textit{J. Pure, Appl. Math.}, {\bf 110}, 259--295 (1996). 
%
%\bibitem[DLM]{DLM2000}
%C. Dong, H.-S. Li and G. Mason, Modular-invariance of trace functions in
%orbifold theory and generalized Moonshine, 
%\textit{Commun. Math. Phys.} \textbf{214} (2000), 1--56.
%
%%\bibitem[DM1]{DongMason94}
%%C.~Dong and G.~Mason, 
%%The construction of the moonshine module as a $\Z_p$-orbifold, 
%%\textit{Contemp. Math.}, {\bf 175}, Amer. Math. Soc., Providence, RI, 37--52 (1994).  
%
%\bibitem[DM]{DongMason97}
%C. Dong and G. Mason, 
%On quantum Galois theory, \textit{Duke Math. J.}, {\bf 86}, no. 2, 305--321 (1997). 
%
%%
%%\bibitem[DN1]{DN1} C.~Dong and K.~Nagatomo, 
%%Classification of irreducible modules for the vertex operator algebra $M(1)^+$, 
%%\textit{J.~Algebra} {\bf216}, 384--404, (1999).
%%
%%\bibitem[DN2]{DongNagatomo99}C.~Dong and K.~Nagatomo, 
%%Representations of vertex operator algebra $V_L^+$ for rank one lattice $L$, 
%%{\it Commun.~Math.~Phys.} {\bf202}, 169--195, (1999).
%%
%%%\bibitem[DN3]{DN3} C.~Dong and K.~Nagatomo, 
%%Classification of irreducible modules for the vertex operator algebra $M(1)^+$ II. Higher Rank, 
%%\textit{J.~Algebra} {\bf240}, 389--325,  (2001).
%
%\bibitem[DRX]{DongRenXu15}
%C.~Dong, L.~Ren, F.~Xu, 
%On orbifold theory, arXiv:1507.03306. 
%
%%\bibitem[DXY]{DongXuYu14}
%%C.~Dong, F.~Xu, N.~Yu, Cyclic permutations of lattice vertex operator algebras,  arXiv:1501.00063. 
%%
%%\bibitem[DY]{DongYamskulna02}
%%C.~Dong, G.~ Yamskulna, 
%%Vertex operator algebras, generalized doubles and dual pairs,
%%{\it Math. Z.}, {\bf 241}, 397--423 (2002). 
%
%%\bibitem[DZ05]{DongZhao05}C.~Dong,  Z.~Zhao, 
%%Modularity in Orbifold Theory for Vertex Operator Superalgebras, 
%%\textit{Commun. Math, Phys.}, {\bf 260}, no. 1, 227--256, (2005). 
%
%\bibitem[Eb]{Ebeling13} 
%W. Ebeling, \textit{Lattice and codes}, Third edition., Springer Spektrum, 2013. 
%
%
%%\bibitem[vE11]{vEkeren11}J.~Ekeren, 
%%Modular Invariance for Twisted Modules over a Vertex Operator Superalgebra, 
%%preprint, arXiv:1111.0682. 
%
%
%\bibitem[EMS]{EMS}
%J.\ van Ekeren, S.\ M\"oller and N.\ Scheithauer, 
%Construction and classification of holomorphic vertex operator algebras,   
%arXiv:1507.08142.
%
%\bibitem[FHL]{FHL} I. Frenkel, Y.-Z. Huang and J. Lepowsky, 
%On axiomatic approaches to vertex operator algebras and modules, 
%\textit{Mem. Amer. Math. Soc.} {\bf104}, (1993).
%
%\bibitem[FLM]{FLM}
%I. Frenkel, J. Lepowsky and A. Meurman,
%\textit{Vertex Operator Algebras and the Monster}, 
%Pure and Appl. Math., Vol. {\bf134}, Academic~Press, Boston, 1988.
%
%%\bibitem[FKS]{FuchsKlemmSchmidt92} 
%%J.~Fuchs, A.~Klemm and M.G.~Schmidt, 
%%Orbifolds by cyclic permutations in Gepner type superstrings in the corresponding Calabi-Yau 
%%manifolds, {\it Ann, Phys.} {\bf 214}, 221--257, (1992).  
%
%%\bibitem[FZ]{FrenkelZhu92}
%%I.~Frenkel and Y.-C.~Zhu,
%%Vertex operator algebras associated to representations of affine and Virasoro algebras,
%%\textit{Duke Math.~J.} {\bf 66}, 123--168, (1992).
%
%%\bibitem[GaKa]{GaKa}  
%%M.~Gaberdiel and H.~Kausch, 
%%A rational logarithmic conformal field theory, 
%%\textit{Phys.~Lett.} {\bf B  386},  no. 1-4, 131--137 (1996).
%
%%\bibitem[GN]{GaberdielNeitzke03} M.~Gaberdiel and A.~Neitzke, 
%%Rationality, quasirationality and finite $W$-algebras,  
%%\textit{Commun.~Math.~Phys.}  {\bf 238},  no. 1-2, 305--331, (2003).
%
%%\bibitem[HK07]{HeluaniKac07}R.~Heliani and V.~Kac, 
%%Supersymmetric vertex algebras, 
%%{\it Commun. Math. Phys.},  {\bf 271}, no. 1, 103--178, (2007).
%
%%\bibitem[H08]{Hohn} G.~H\"{o}hn, 
%%Conformal designs based on vertex operator algebras, 
%%{\it Adv. in Math.}, {\bf 217}, No. 5, 2301--2335, (2008).
%
%%\bibitem[HLZ10]{HuangLepowskyZhang10} 
%%Y.-Z. Huang, J. Lepowsky and L, Zhang, 
%%Logarithmic tensor category theory, II:Logarithmic formal calculus and properties of logarithmic 
%%intertwining operators, arXiv:1012.4196. 
%
%%\bibitem[HY11]{HY}Y.-Z.~Huang, J.~Yang, 
%%Logarithmic intertwining operators and associative algebras, 
%%{\it J. Pure and Applied Algebra}, {\bf 216}, No. 6, 1467--1492, (2012).
%
%%\bibitem[H]{Huang1} Y.-Z.~Huang, 
%%Vertex operator algebras and the Verlinde conjecture, 
%%{\it Commun. Contemp. Math.} {\bf 10}, no. 1, 103--154, (2008).
%
%%\bibitem[H]{Huang07} Y.-Z.~Huang, Cofiniteness conditions, projective covers and the logarithmic tensor product theory, {\it J. Pure Appl. Algebra} {\bf 213}, no. 4, 458--475,  (2009).
%
\bibitem[G]{Gorenstein}
D.~Gorenstein, \textit{Finite groups}, 2nd ed. 
Chelsea Publishing Co., New York (1980).

\bibitem[Ha]{Harada}
K. Harada, 
Revisiting Character Theory of Finite Groups, \textit{Bulletin of the Institute of Mathematics}, Academia Sinica (New Series)
{\bf 13} (2018), 383--395. 


\bibitem[Hi]{Hida}
A. Hida, 
Harada\rq{}s conjecture on character degrees and class sizes -symmetric and alternating groups-, \textit{Research on algebraic combinatorics and representation theory of finite groups and vertex operator algebras}, RIMS kokyuroku (Japanese), RIMS {\bf 2086} (2018), 144--153. 

\bibitem[Hu]{Huppert} 
B.~Huppert, 
\textrm{Representations and characters of groups}, \textit{Cambridge Mathematical Textbooks}, Cambridge University Press, 1993.


\bibitem[I]{Isaacs} 
I. Martin Isaacs, 
\textit{Character theory of finite groups}, Dover edition 1994, originally 
published in  Pure and applied mathematics, Academic Press, 1976.


\bibitem[JL]{JamesLiebeck} 
G. James and Martin Liebeck, 
\textit{Representations and characters of groups}, Cambridge Mathematical Textbooks, Cambridge University Press, 1993.

%%\bibitem[K]{Kac98} V.~Kac, \textit{Vertex algebras for beginners}, Second edition, University Lecture Series {\bf 10}, American Mathematical Society, Providence, RI, 1998.
%
%%\bibitem[Kau]{K1} H.~Kausch, Curiosities at $c=-2$, hep-th/9510149.
%
%%\bibitem[KS]{KlemmSchmidt90} A.~Klemm and M.G.~Schmidt, Orbifolds by cyclic permutations of tensor product conformal field theories, {\it Phys. Lett.} {\bf B245}, 53--58, (1990).
%
%%\bibitem[KW94]{KacWang94}, V.~Kac and W.~Wang, Vertex operator superalgebras and their representations, Mathematical aspects of conformal and topological field theories and quantum groups (South Hadley, MA, 1992), {\it Contemp. Math.}, {\bf 175}, 161--191,  Amer. Math. Soc., Providence, RI, (1994).
%
%\bibitem[GL]{GriessLam11}
%R. L. Griess, Jr. and C.H. Lam, 
%A moonshine path for $5A$ and associated lattices of ranks $8$ and $16$, 
%\textit{J. Algebra} \textbf{331} (2011), 338--361.


\bibitem[K]{Kiyota} 
M.~Kiyota, 
\textit{Harada conjecture II and its block refinement}, 
\textit{Cohomology theory of finite groups and related topics}, RIMS kokyuroku (Japanese), RIMS {\bf 2061} (2018), 56--60. 

%
%\bibitem[L]{Lepowsky85} 
%J. Lepowsky,  
%Calculus of twisted vertex operators, 
%\textit{Proc. Natl. Acad. Sci. USA} {\bf 82}, 8295--8299, 1985. 
% 
%\bibitem[LL]{LepowskyLi04}J.~Lepowsky and H.-S.~Li, 
%\textit{Introduction to vertex operator algebras and their representations}, 
%Prog.~Math., Birkh{\"a}user, 2004.  
%
%%\bibitem[Li99-1]{Li99-1} H. -S. Li, Determining fusion rules by $A(V)$-modules and bimodules,\textit{J.~Algebra} {\bf 212}, 515--556, (1999).
%
%
%%\bibitem[Li1]{Li1} H.-S.~Li, Symmetric invariant bilinear forms on vertex operator algebras, \textit{J.~Pure.~and~Appl.~Algebra} {\bf 96}, Issue 3 , 279--297 (1994).
%
%%\bibitem[Li2]{Li2} H.-S.~Li, Local systems of vertex operators, vertex superalgebras and modules, \textit{J.~Pure. and Appl.~Algebra} \textbf{109} (1996), 143--195.
%
%%\bibitem[Li99-2]{Li99-2} H.-S.~Li,  Some finiteness properties of regular vertex operator algebras, \textit{J.~Algebra} {\bf 212} (1999), 495--514.
%
%%\bibitem[Li01]{Li01} H. -S. Li, The Regular Representation, Zhu's $A(V)$-theory, and Induced Modules,\textit{J.~Algebra} {\bf 238}, No. 1, 159--193, (2001).
%
%%\bibitem[Li02]{Li02}H. -S. Li, Regular representations of vertex operator algebras,\textit{Commun. Contemp. Math.}, {\bf 4}, No. 4, 639--{???}, (2002).
%
%%\bibitem[Li]{Li04} H.-S.~Li,  Vertex algebras and vertex Poisson algebras. {\it Commun. Contemp. Math.} {\bf 6}, no. 1, 61--110, (2005).
%
%%\bibitem[Li2]{Li05} H.-S.~Li,  Abelianizing vertex algebras. {\it Comm. Math. Phys.}, {\bf 259}, no. 2, 391--411, (2005).
%
%
%\bibitem[MN]{MatsuoNagatomo99} A.~Matsuo and K.~Nagatomo,  
%Axioms for a vertex algebra and the locality of quantum fields, 
%\textit{MSJ Memoirs} {\bf 4}, Mathematical Society of Japan, (1999).
%
%%\bibitem[MNT]{MatsuoNagatomoTsuchiya10}
%%A.~Matsuo, K.~Nagatomo, and A.~Tsuchiya,
%%Quasi-finite algebras graded by Hamiltonian and vertex operator algebras.
%%In \textit{Moonshine: the first quarter century and beyond}.
%%London Mathematical Society Lecture Note Series 372. Cambridge: Cambridge University Press, 2010
%%
%%%\bibitem[Mil02]{Milas02} A.~Milas, Weak modules and logarithmic intertwining operators for vertex operator algebras in Recent developments in infinite-dimensional Lie algebras and conformal field theory (Charlottesville, VA, 2000),  201--225, \textit{Contemp. Math.}, {\bf 297}, Amer. Math. Soc., Providence, RI, 2002.
%%
%%
%%
%%\bibitem[M]{Miyamoto04} 
%%M.~Miyamoto, 
%%Modular invariance of vertex operator algebras satisfying $C\sb 2$-cofiniteness, 
%%\textit{Duke~Math.~J.} {\bf122},  no. 1, 51--91,  (2004).
%
%\bibitem[M1]{Miyamoto04} M.~Miyamoto, 
%A new construction of the Moonshine vertex operator algebra over the real number field, 
%\textit{Ann. of Math.}  {\bf 159},  no. 2, 535--596 (2004).
%
%%\bibitem[M2]{Miyamoto10} M.~Miyamoto, 
%%A $\Z_3$-orbifold theory of lattice vertex operator algebra and $\Z_3$-orbifold constructions 
%%in Symmetries, integrable systems and representations, 319--344, 
%%\textit{Springer Proc. Math. Stat.}, {\bf 40}, Springer, Heidelberg, 2013. 
%
\bibitem[M]{Miyamoto} M.~ Miyamoto,  Groups of nilpotent class $2$ and Harada's conjecture II,  RIMS kokyuroku (Japanese), RIMS {\bf 2287} (2023).

%\bibitem[M2]{Miyamoto15} M. Miyamoto, $C_2$-cofiniteness of cyclic-orbifold Models, 
%{\it Commun. Math. Phys.}, {\bf 335}, No. 3, 1279--1286 (2015). 
%
%%\bibitem[M2]{Miya98} Miyamoto, Masahiko Representation theory of code vertex operator algebra.  {\it J. Algebra}  {\bf 201},  no. 1, 115--150,   (1998).
%
%%\bibitem[Miya11]{Miyamoto11} M.~Miyamoto, Flatness and Semi-Rigidity of Vertex Operator Algebras,  arXiv:1104.4675, (2011).
% 
%%\bibitem[MT]{MiyamotoTanabe04} M. Miyamoto, K. Tanabe, Uniform product of $A_{g,n}(V)$ for an orbifold
%%model $V$ and $G$-twisted Zhu algebra,
%%{\it J. Alg.}, {\bf 274}, 80--96 (2004). 
%
%\bibitem[M\"{o}]{Moller16} 
%S.~M\"{o}ller, 
%A Cyclic Orbifold Theory for Holomorphic Vertex Operator Algebras and Applications, 
%arXiv:1611.09843, Ph. D. Thesis. (2016). 
% 
%%\bibitem[NT]{NT} K.~Nagatomo and A.~Tsuchiya, Conformal field theories associated to regular chiral vertex operator algebras I: theories over the projective line,  \textit{Duke~Math.~J.}  {\bf 128},  no. 3, 393--471, (2005).
%
%%\bibitem[Rot08]{Rotman}J. Rotman, \textit{An introduction to homological algebra}, Springer New York, 2nd edition,  2008. 
%
%%\bibitem[Run12]{Runkel12}I. Runkel, A braided monoidal category for free super-bosons, arXiv:1209.5554v1. 
%
%\bibitem[Sh]{Shimakura11} H.~Shimakura, An $E_8$-approach to the moonshine vertex operator
%algebra, \textit{J. London Math. Soc.}, {\bf 83} 493--516 (2011).
%
\bibitem[Sug]{Sugimoto} M.~Sugimoto, 
Harada's conjecture II for the finite general linear groups and unitary groups, arXiv:2304.10708, (2023). 

\bibitem[Suz1]{Suzuki1} M.~Suzuki, \textrm{Group Theory I}, 
\textit{Grundlehren der Mathematischen Wissenschaften}, {\bf 247}, Springer-Verlag, Berlin-New York (1982).

\bibitem[Suz2]{Suzuki2} M.~Suzuki, \textrm{Group Theory II}, 
\textit{Grundlehren der Mathematischen Wissenschaften}, {\bf 248}, 
Springer-Verlag, Berlin-New York, (1986).


%\bibitem[TY]{TanabeYamada13}
%K. Tanabe, and H. Yamada, 
%Fixed point subalgebras of lattice vertex operator algebras by an automorphism of order three,
%\textit{J. Math. Soc. Japan}, {\bf 65}, no. 4, 1169--1242 (2013). 
%
%%\bibitem[Yama]{Yamauchi04}
%%H. Yamauchi, Module categories of simple current extensions of vertex operator algebras, \textit{J. Pure Appl. Algebra}, {\bf 189}, no. 1-3, 315--328 (2004).  
%%
%%\bibitem[Yams]{Yamskulna04}
%%G.~ Yamskulna, 
%%$C_2$-cofiniteness of the vertex operator algebra $V_L^+$ when $L$ is a rank one lattice,
%%\textit{Comm. Algebra},  {\bf 32}, no. 3, 927--954 (2004). 
%%
%
%%\bibitem[Xu]{Xu98}
%%X.~Xu,
%%\textit{Introduction to vertex operator superalgebras and their modules},
%%Mathematics and its applications, Kluwer Academic Publishers, 1998.
%
%\bibitem[Z]{Zhu96}
%Y.-C.~Zhu, 
%Modular invariance of characters of vertex operator algebras, 
%\textit{J.~Amer.~Math.~Soc.} {\bf9}, 237--302,  (1996).
\end{thebibliography}
\end{document}